\documentclass[leqno,12pt]{article} %leqno is the option to put formula numbers on the left side
\usepackage{amsmath, amssymb}
\usepackage{amsthm} %theorem environment option
\theoremstyle{plain} %text of this environment is typesetted in italics
\newtheorem{theorem}{\indent\sc Theorem}[section]
\newtheorem{lemma}[theorem]{\indent\sc Lemma}
\newtheorem{corollary}[theorem]{\indent\sc Corollary}
\newtheorem{proposition}[theorem]{\indent\sc Proposition}

\theoremstyle{definition} %text of this environment is typesetted in roman letters
\newtheorem{definition}[theorem]{\indent\sc Definition}
\newtheorem{remark}[theorem]{\indent\sc Remark}

\newcommand{\dbar}{\overline{\partial}}
\newcommand{\pa}{\partial}
\newcommand{\ov}{\overline}
\newcommand{\xu}{\sqrt{-1}}

\newcommand{\Res}{\text{Res}}

\makeatletter
\def\address#1#2{\begingroup
\noindent\parbox[t]{7.8cm}{%
\small{\scshape\ignorespaces#1}\par\vskip1ex
\noindent\small{\itshape E-mail address}%
\/: #2\par\vskip4ex}\hfill%
\endgroup}%
\makeatother
\title{\uppercase{Logarithmic Akizuki--Nakano vanishing theorems on weakly pseudoconvex K\"ahler manifolds}} %title of the paper
\author{
\bigskip \\
\textsc{Yongpan Zou} %names of authors
}
\date{} %leave empty
\begin{document}

\maketitle

%%%%%%%%%%%%%%% footnote %%%%%%%%%%%%%%%%
\footnote{ %2020 MSC numbers
2020 \textit{Mathematics Subject Classification}.
Primary 32E05; Secondary 32E30, 32T35.
}
\footnote{ %key words and phrases
\textit{Key words and phrases}.
Logarithmic vanishing theorem, weakly pseudoconvex K\"ahler manifold, Runge approximation.
}
\footnote{ %acknowledgment of support etc. if any
%$^{*}$Thanks.
}
%%%%%%%%%%%%%%%%%%%%%%%%%%%%%%%%%%%%%%%%%

\begin{abstract}
In this paper, we establish a logarithmic vanishing theorem on weakly pseudoconvex K\"ahler manifolds, where the divisor may have infinitely many irreducible components. This result serves as a generalization of Norimatsu's findings on compact Kähler manifolds. We derive vanishing theorems for certain direct image sheaves as a direct corollary.
\end{abstract}

\section{Introduction}
One of the central topics in complex and algebraic geometry is the cohomological vanishing theorem. The famous Akizuki--Nakano vanishing theorem shows that if $F$ is a positive line bundle over an $n$-dimensional compact K\"ahler manifold $X$, then
$$
H^q(X, \Omega^p_X\otimes F) =0 \quad\text{for any} \quad p+q \geq n+1.
$$

The generalization of the Akizuki--Nakano vanishing theorem on weakly pseudoconvex or weakly $1$-complete K\"ahler manifolds has been accomplished by Nakano \cite{Nakano73, Nakano74}, Kazama \cite{Kazama73}, Abdelkader \cite{Abdelkader80}, Takegoshi \cite{Take81}, Ohsawa--Takegoshi \cite{TaOh81}, and others.

In a different vein, Norimatsu obtained the logarithmic vanishing theorem on compact K\"ahler manifolds in \cite{Norimatsu78}. Esnault and Viehweg, in \cite{EsVi86}, delved into the logarithmic de Rham complexes and vanishing theorems on projective manifolds. They derived logarithmic type vanishing theorems for the pair $(X, D)$, where $X$ is a projective manifold and $D$ is a simple normal crossing divisor. Their methods are grounded in Hodge theory and the degeneration of the Hodge to de Rham spectral sequence.

More recently, in \cite{HLWY16}, Huang--Liu--Wan--Yang obtained corresponding results on compact K\"ahler manifolds using standard analytic techniques such as the $L^2$-method.

In this paper, we aim to generalize the results of Norimatsu, Esnalut--Viehweg, and Huang--Liu--Wan--Yang to open pseudoconvex K\"ahler manifolds. When dealing with weakly pseudoconvex manifolds, it is important to consider the possibility that the divisor may have infinitely many irreducible components. On the other hand, by definition, a divisor has only locally finite irreducible components. As a result, within any relatively compact subset of the manifold, only a finite number of divisor components have non-empty intersection with this subset.

For compact K\"ahler manifolds or projective manifolds, the simple normal crossing divisor must have finitely many irreducible components. In contrast, on non-compact manifolds such as the weakly pseudoconvex K\"ahler manifolds studied in this paper, this assertion no longer holds.
We initially focus on the case where the divisor has only finitely many components. This can be deduced from the classical Akizuki--Nakano vanishing on weakly pseudoconvex K\"ahler manifolds and Deligne's filtration (for details, refer to the end of Section $2$). Specifically, we obtain:
\begin{theorem}  \label{main-finite}
Let $X$ be an $n$-dimensional weakly pseudoconvex K\"ahler manifold and $F$ be a positive line bundle on $X$. Let $D$ be a simple normal crossing divisor with finitely many irreducible components on $X$. We have
$$
H^q(X, \Omega^p_X(\log D)\otimes F) = 0
$$
for any $p+q\geq n+1$.
\end{theorem}

Then, starting from Section $3$, we shift our focus to the general case, where the number of irreducible components of the divisor may be infinite. Unless explicitly stated otherwise, this remains the default setting throughout this paper. The following two theorems represent our main results.

\begin{theorem} [=Corollary \ref{vanofholo}] \label{main}
Let $X$ be a holomorphically convex K\"ahler manifold and $F$ be a positive line bundle on $X$. Let $D$ be a simple normal crossing divisor on $X$, here the number of irreducible components of the divisor may be infinite. We have
$$
H^q(X, \Omega_X^p(\log D)\otimes F) = 0 \quad \text{for any} \quad p+q \geq n+1.
$$
\end{theorem}

The crucial case arises when $p=n$, thereby yielding $\Omega^n_X(\log D) = K_X\otimes \mathcal{O}_X(D)$. In this scenario, we extend this result to weakly pseudoconvex K\"ahler manifolds, leading to
\begin{theorem} [= Theorem \ref{q2} + Theorem \ref{q1}] \label{main2}
Let $X$ be a $n$-dimensional weakly pseudoconvex K\"ahler manifold and $F$ is a positive line bundle on $X$. Let $D$ be a simple normal crossing divisor on $X$, here the number of irreducible components of the divisor may be infinite. We have
$$
H^q(X, K_X\otimes \mathcal{O}_X(D)\otimes F) = 0.
$$
for any $q\geq 1$.
\end{theorem}

%get the generalization of Norimatsu's vanishing theorem %\cite{Norimatsu78} and Le Potier's vanishing theorem %\cite{Potier75}. On projective algebraic manifolds, these results %have close relationships with the famous Kawamata--Viehweg %vanishing theorem.

In Theorem \ref{main2}, there is no need to twist the sheaf $K_X\otimes \mathcal{O}_X(D)\otimes F$ with the multiplier ideal sheaf $\mathcal{I}(D)$ of the divisor $D$. The vanishing of $H^q(X, K_X\otimes \mathcal{O}_X(D)\otimes F \otimes \mathcal{I}(D))$ is a direct consequence of Nadel's vanishing theorem, see \cite[Theorem $5.11$]{Dem12a}. Our approach combines the $L^2$ technique from \cite{HLWY16} with a Runge-type approximation method rooted in \cite{Nakano74, Kazama73, Take81, TaOh81}.

The approach to generalizing the vanishing theorems from compact K\"ahler manifolds to weakly pseudoconvex manifolds involves the study of their relatively compact sublevel subsets. For a weakly pseudoconvex K\"ahler manifold $X$ with a smooth plurisubharmonic exhaustion function $\Phi$ and a sequence of positive real numbers tending to infinity, we consider each sublevel subset $X_c:= \{x\in X: \Phi(x) < c\}$, which is relatively compact in $X$ and constitutes a weakly pseudoconvex manifold itself. Using the $L^2$ technique, we establish the logarithmic vanishing theorem on each sublevel subset. Note that we can also apply Theorem \ref{main-finite} to derive the vanishing theorem on each sublevel subset, as within each relatively compact sublevel manifold, the components of the divisor $D$ are finite.

The collection of these sublevel subsets forms a Leray covering of $X$, and therefore, we can focus on the \v{C}ech cohomology. Through the approximation process, we derive the global vanishing theorem.

One of the motivations for studying cohomology on weakly pseudoconvex K\"ahler manifolds is the ability to investigate the corresponding higher direct image sheaves. As a direct corollary, we obtain
\begin{corollary} [Corollary $4.1$]\label{direct}
Let $f: X \rightarrow S$ be a proper holomorphic morphism from a K\"ahler manifold $X$ onto the reduced complex space $S$. Let $D$ be a simple normal crossing divisor. If $F$ is a positive holomorphic line bundle on $X$, then
$$
 R^q f_{\ast}(\Omega^p_X(\log D)\otimes F) =0 \quad \text{for any} \quad p+q \geq n+1.
$$
\end{corollary}
In Theorem \ref{main2} above, if we replace the positive line bundle with a positive vector bundle in the sense of Griffiths (Nakano), it would be interesting to establish the corresponding vanishing results. Additionally, as Professor Ohsawa pointed out in \cite{Ohsa21}, there might be interest in generalizing the results from \cite{LRW19} and \cite{LWY19} to the weakly pseudoconvex situation.

\noindent\textbf{Acknowledgement}: The author expresses sincere appreciation to Professor Shigeharu Takayama for his guidance and warm encouragement, as well as Professor Sheng Rao for his support. The author is also grateful to the University of Tokyo for the Special Scholarship for International Students (Todai Fellowship).

\section{Preliminaries}
In this section, we introduce some basic definitions and results in complex geometry. Unless stated otherwise, $X$ denotes a complex manifold of dimension $n$. The primary reference is \cite{Dem12}.
\begin{definition} [Chern connection and curvature form of vector bundle]
Let $(E, h)$ be a holomorphic vector bundle on $X$. Corresponding to this metric $h$, there exists the unique Chern connection $D=D_{(E,h)}$, which can be split in a unique way as a sum of a $(1,0)$ and a $(0,1)$ connection, i.e., $D=D'_{(E,h)} + D''_{(E,h)}$.
Furthermore, the $(0,1)$ part of the Chern connection $D''_{(E,h)} =\dbar$. The curvature form is defined to be $\Theta_{E,h} := D^2_{(E,h)}$.
On a coordinate patch $\Omega \subset X$ with complex coordinate $(z_1,\ldots,z_n)$, denote by $(e_1,\ldots,e_r)$ an orthonormal frame of vector bundle $E$ with rank $r$. Set
$$
\xu \Theta_{E,h} = \xu \sum_{1\leq j,k\leq n, 1\leq \lambda,\mu \leq r} c_{jk\lambda\mu} dz_j\wedge d\ov{z}_k \otimes e^{\ast}_{\lambda} \otimes e_{\mu} , \quad c_{jk\mu\lambda}=\ov{c}_{jk\lambda\mu}.
$$
Corresponding to $\xu\Theta_{E,h}$, there is a Hermitian form $\theta_{E,h}$ on $TX\otimes E$ defined by
$$
\theta_{E,h}(\phi,\phi) = \sum_{jk\lambda\mu} c_{jk\lambda\mu}(x) \phi_{j\lambda}\ov{\phi}_{k\mu}, \quad \phi=\sum_{j,\lambda}\phi_{j\lambda}\frac{\partial}{\partial{z_j}}\otimes e_{\lambda} \in TX\otimes E.
$$
\end{definition}

\begin{definition} [Positive vector bundle]
A holomorphic vector bundle $(E,h)$ is said to be
\begin{enumerate}
    \item Nakano positive (resp. Nakano semi-positive) if for every nonzero tensor $\phi \in TX\otimes E$, we have
$$
 \theta_{E,h}(\phi,\phi) >0 \quad(\text{resp.} \geq 0).
$$
  \item Griffiths positive (resp. Griffiths semi-positive) if for every nonzero decomposable tensor $\xi\otimes e \in TX\otimes E$, we have
$$
 \theta_{E,h}(\xi\otimes e, \xi\otimes e) >0 \quad(\text{resp.} \geq 0).
$$
\end{enumerate}
It is clear that Nakano positivity implies Griffiths positivity, and both concepts coincide if $r=1$. In the case of a line bundle, $E$ is simply referred to as positive (or semi-positive).
\end{definition}

\begin{definition} [Singular metric and curvature current on line bundle]
Let $(F,h)$ be a holomorphic line bundle on complex manifold $X$ endowed with possible singular Hermitian metric $h$. For any given trivialization $\theta : F|_{\Omega} \simeq \Omega \times \mathbb{C}$ by
$$ \| \xi \|_h = |\theta(\xi)| e^{-\phi(x)},  \quad x \in \Omega, \xi \in F_x,
$$
where $\phi \in L^1_{loc}(\Omega)$ is a weight function of the metric. The curvature current $\sqrt{-1} \Theta_h(F)$ of $h$ is defined by
$$  \sqrt{-1} \Theta_h(F) = \sqrt{-1} 2\partial \overline{\partial} \phi.
$$
The Levi form $\sqrt{-1}\pa\dbar \phi$ is taken in the sense of distributions and thus the curvature is a $(1,1)$-current but not always a smooth $(1,1)$-form. It is globally defined on $X$ and independent of the choice of trivializations. The curvature $\sqrt{-1} \Theta_h(F)$ of $h$ is said to be positive if $\sqrt{-1} \Theta_h(F) \geq 0$ in the sense of current.
\end{definition}

\begin{definition} [Psh function and quasi-psh function]
A function $u: \Omega \rightarrow [-\infty, \infty)$ defined on a open subset $\Omega \in \mathbb{C}^n$ is called plurisubharmonic (psh, for short) if
\begin{enumerate}
\item $u$ is upper semi-continuous;
\item for every complex line $Q \subset \mathbb{C}^n$, $u|_{\Omega \cap Q}$ is subharmonic on $\Omega \cap Q$.
\end{enumerate}
A quasi-plurisubharmonic (quasi-psh, for short) function is a function $v$ which is locally equal to the sum of a psh function and a smooth function.
\end{definition}

\begin{definition} [Multiplier ideal sheaves]
Let $\phi$ be a quasi-psh function on a complex manifold $X$, the multiplier ideal presheaf $\mathcal{I}'(\phi) \subset \mathcal{O}_X$ is defined by 
$$ \Gamma(U, \mathcal{I}'(\phi)) = \{f \in \mathcal{O}_X(U) :\quad |f|^2e^{-2\phi} \in L^1_{loc}(U) \}
$$
for every open set $U \subset X$. The multiplier ideal sheaf $\mathcal{I}(\phi) \subset \mathcal{O}_X$ is the corresponding sheafification of this presheaf. 
For a line bundle $(F, h)$, if the local weight of the metric $h$ is $\phi$, we denote the multiplier ideal sheaf interchangeably as  $\mathcal{I}(h)$ or $\mathcal{I}(\phi)$.
\end{definition}

We now turn to the introduction of some basic definitions for weakly pseudoconvex manifolds.
\begin{definition} [Weakly pseudoconvex = Weakly $1$-complete]
A function $\phi: X\rightarrow [-\infty, +\infty)$ on a manifold $X$ is said to be exhaustive if all sublevel sets
$$
X_c:=\{x\in X: \phi(x)< c \} \quad c < \sup \phi,
$$
are relatively compact. A complex manifold $X$ is called \emph{weakly pseudoconvex} if there exists a smooth plurisubharmonic exhaustion function $\phi: X\rightarrow \mathbb{R}$ with $\sup \phi = +\infty$. Similarly, a complex manifold $X$ is said to be strongly pseudoconvex if the exhaustion function is smooth and strictly plurisubharmonic.
\end{definition}

\begin{proposition} \cite{Nakano70, Dem12} \label{hatom}
Every weakly pseudoconvex K\"ahler manifold carries a complete K\"ahler metric.
\end{proposition}
\begin{proof}
We show the proof of this because we will use it later. Let $(X, \omega)$ be a K\"ahler manifold with K\"ahler form $\omega$, and $\phi$ be an exhaustive psh function on $X$. Set $\hat{\omega}= \omega + \xu\pa\dbar(\chi\circ\phi)$, where $\chi$ is a smooth convex increasing function. Then
\begin{align*}
\hat{\omega} &= \omega + \xu (\chi'\circ\phi)\pa\dbar\phi + \xu(\chi''\circ\phi)\pa\phi\wedge\dbar\phi \\
&\geq \omega + \xu \pa(\rho\circ\phi)\wedge\dbar(\rho\circ\phi)
\end{align*}
where $\rho=\int_0^t \sqrt{\chi''(u)}du$. We thus have complete metric $\hat{\omega}$ as soon as $\lim_{t\rightarrow+\infty}\rho(t)=+\infty$, i.e.
$$
\int_0^{+\infty} \sqrt{\chi''(u)}du = +\infty.
$$
\end{proof}

The complete K\"ahler metric is crucial in solving the $\dbar$-equation. The next two fundamental theorems regarding the $L^2$-estimate of the $\dbar$-equation are of paramount importance for this paper.

\begin{theorem}[$\dbar$-equation with a metric that is not necessarily complete] \cite[\rm Chapter \rm VIII, Theorem $6.1$]{Dem12} \label{dbareq}
Let $X$ be a complete K\"ahler manifold with a K\"ahler metric $\omega$ which is not necessarily complete. Let $(E,h)$ be a Hermitian vector bundle of rank $r$ over $X$, and assume that the curvature operator $B:=[\xu\Theta_{E,h},\Lambda_\omega]$ is semi-positive definite everywhere on $\wedge^{n,q}T_X^*\otimes E$, for some $q\geq 1$. Then for any form $g\in L^2(X,\wedge^{n,q}T^*_{X}\otimes E)$ satisfying $\bar{\partial}g=0$ and $\int_X\langle B^{-1}g,g\rangle dV_\omega<+\infty$, there exists $f\in L^2(X,\wedge^{n,q-1}T^*_X\otimes E)$ such that $\bar{\partial}f=g$ and
$$\int_X|f|^2dV_\omega\leq \int_X\langle B^{-1}g,g\rangle dV_\omega.$$
\end{theorem}

\begin{theorem}[$\dbar$-equation with complete metric] \cite[Theorem $5.1$]{Dem12a} \label{dbareq-2}
Let $X$ be a complete K\"ahler manifold with a complete K\"ahler metric $\omega$. Let $(E,h)$ be a Hermitian vector bundle of rank $r$ over $X$, and assume that the curvature operator $B:=[\xu\Theta_{E,h},\Lambda_\omega]$ is semi-positive definite everywhere on $\wedge^{p,q}T_X^*\otimes E$, for some $q\geq 1$. Then for any form $g\in L^2(X,\wedge^{p,q}T^*_{X}\otimes E)$ satisfying $\bar{\partial}g=0$ and $\int_X\langle B^{-1}g,g\rangle dV_\omega<+\infty$, there exists $f\in L^2(X,\wedge^{p,q-1}T^*_X\otimes E)$ such that $\bar{\partial}f=g$ and
$$\int_X|f|^2dV_\omega\leq \int_X\langle B^{-1}g,g\rangle dV_\omega.$$
\end{theorem}

\begin{definition} [Holomorphically convex manifold]
A complex manifold $X$ is called \emph{holomorphically convex} if for any compact set $K \subset X$, its holomorphic hull $\hat{K}= \{ x\in X: |f(x)| \leq \underset{K}{\text{sup}}|f| ~\text{for all} ~ f\in \mathcal{O}_X(X) \}$ is compact too.
\end{definition}

\begin{definition}[Stein manifold]
A complex manifold $X$ is called \emph{Stein manifold} if $X$ is holomorphically convex and for any $x,y\in X, x\neq y$, there exists a $f\in \mathcal{O}_X(X)$ with $f(x) \neq f(y)$.
\end{definition}

While every holomorphically convex manifold is weakly pseudoconvex, the converse does not hold. In the case of a holomorphically convex manifold, we have the classical Remmert reduction, which establishes its connection to Stein space.
\begin{remark} [Remmert reduction] \label{remmert}
If $X$ is a holomorphically convex manifold, then by Remmert reduction, there exists a normal Stein space $S$ and a proper, surjective, holomorphic morphism $f: X \rightarrow S$ such that
\begin{enumerate}
    \item $f_{\ast}\mathcal{O}_X = \mathcal{O}_S$,
    \item $f$ has connected fibers,
    \item The map $f^{\ast}: \mathcal{O}_S(S) \rightarrow \mathcal{O}_X(X)$ is an isomorphism,
    \item The pair $(f, S)$ is unique up to biholomorphism.
\end{enumerate}
\end{remark}

\begin{remark} [Sublevel subset of weakly pseudoconvex manifolds] \label{sublevel}
Let $(X, \omega, \Phi)$ be a weakly pseudoconvex K\"ahler manifold with smooth psh exhaustion function $\Phi$. Without loss of generality, we may assume $\Phi$ is positive. For any positive real number $c$, the sublevel set $X_c = \{ x\in X : \Phi(x) < c \}$ is relative compact in $X$ and again pseudoconvex with respect to the exhaustion function $ \Phi_c := \frac{1}{c-\Phi}$. Set $\omega_c:= \omega|_{X_c}$, then $(X_c, \omega_c, \Phi)$ is again a weakly pseudoconvex K\"ahler manifold, and thus we have an exhaustion sequence of pseudoconvex sublevel set $(X_c, \omega_c, \Phi_c)$. Let $D= \sum_{i\in \Lambda} D_i$ be the simple normal crossing divisor on $X$, here the index set $\Lambda$ may be infinite. Set $\Lambda_c := \{i \in \Lambda : D_i \cap X_c \neq \emptyset\}$ and $D_c := \sum_{i\in \Lambda_c} D_i$, it follows that $\Lambda_c$ is a finite set and hence $D_c$ has only finitely many components.
\end{remark}

Now we introduce some basic properties of the sheaf of logarithmic differential forms and the logarithmic integrable connections on complex manifolds were developed by Deligne in \cite{De70}, Esnault and Viehweg in \cite{EsVi86} studied the relations between logarithmic de Rham complexes and vanishing theorems on projective manifolds.
Let $X$ be a complex manifold and $D$ be a simple normal crossing divisor on it, i.e., $D=\sum_{i} D_i$, where each $D_i$ are distinct smooth hypersurfaces intersecting transversely in $X$. That means at each point $x\in X$, there are at most $n$ divisors $D_i$ passing through $x$. In particular, given $x$, one can find complex analytic coordinates $(z_1,\ldots,z_n)$ in a neighborhood $U$ of $x$, such that the local equation of $D\cap U$ in $U$ is $z_1\cdots z_s=0$, $s$ depending on $x$. Hence even if the number of components of $D$ is infinite, we can still define its logarithmic forms. A logarithmic differential $(k,0)$-form $\alpha$, in a sufficiently small neighborhood $U$ of any $x\in D$ can be written as
$$
\alpha = \sum \alpha_{I,J} \frac{dz_{i_1}}{z_{i_1}}\wedge \cdots \wedge \frac{dz_{i_l}}{z_{i_l}} \wedge dz_{j_1}\wedge \cdots \wedge dz_{j_m}.
$$
Here $I =\{i_1,\ldots, i_l\} \subseteqq \{1,\ldots,s\}, J=\{j_1,\ldots,j_m\}\subseteqq \{s+1,\ldots,n\}$ and $l+m=k$, all $\alpha_{I,J}$ are smooth functions on $U$.
We are mainly interested in those holomorphic logarithmic forms, i.e., all $\alpha_{I, J}$ are holomorphic functions. The sheaf of germs of holomorphic $p$-forms on $X$ with at most logarithmic poles along $D$, we denote it by $\Omega^p_{X}(\log D)$. 
%Its space of sections on any open subset $W$ of $X$ is
%$$
%\Gamma(W, \Omega^p_{X}(\log D)) := \{ \alpha \in \Gamma(W, \Omega^p_{X} \otimes \mathcal{O}_X(D)) : d\alpha \in \Gamma(W, \Omega^{p+1}_{X} \otimes \mathcal{O}_X(D)) \}.
%$$
On any open submanifold $X_c$, by the definition of logarithmic forms, we know the restriction $\Omega_X^p(\log D)|_{X_c} = \Omega^p_{X_c}(\log D_c)$.

%\begin{definition} [Poincar\'e type metric] \label{Poincare type metric}
%Denote by $j: Y=X\backslash D \rightarrow X$ the natural inclusion and we can choose a local coordinate chart $(W; z_1, \cdots, z_n)$ of $X$ such that the locus of $D$ is given by $z_1\cdots z_s =0$ and $Y\cap W= W_r^{\ast}= (\Delta_r^{\ast})^s \times (\Delta_r)^{n-s}$ where $\Delta_r$ (resp. $\Delta_r^{\ast}$) is the (resp. punctured) open disk of radius $r$ in the complex plane. We say that the metric $\omega_Y$ on $Y$ is of Poincar\'e type along $D$, if for each local coordinate chart $(W; z_1, \cdots, z_n)$ along $D$ the restriction $\omega_Y|_{W_r^{\ast}}$ is equivalent to the usual Poincar\'e type metric $\omega_P$ defined by
%\begin{equation}
%\omega_P = \sqrt{-1} \sum_{j=1}^s \frac{dz_j \wedge d \ov{z}_j}{|z_j|^2\cdot \log^2|z_j|^2} + \sqrt{-1} \sum_{j=s+1}^n dz_j \wedge d \ov{z}_j.
%\end{equation}
%\end{definition}

\begin{definition} [Poincar\'e type metric] \label{Poincare type metric}
On complex manifold $X$ with simple normal divisor $D$, we denote the complement of $D$ by $Y$. We can select a local coordinate chart $(W; z_1, \ldots, z_n)$ of $X$ such that the locus of $D$ is given by $z_1\cdots z_s =0$ and $Y\cap W \simeq (\Delta^{\ast})^s \times (\Delta)^{n-s}$ where $\Delta$ (respectively $\Delta^{\ast}$) represents the open disk (respectively punctured) in the complex plane. We say that the metric $\omega$ on $Y$ is of Poincar\'e type along $D$, if for each local coordinate chart the restriction $\omega|_{Y\cap W}$ is equivalent to the usual Poincar\'e type metric $\omega_P$ defined by
\begin{equation}
\omega_P = \sqrt{-1} \sum_{j=1}^s \frac{dz_j \wedge d \ov{z}_j}{|z_j|^2\cdot \log^2|z_j|^2} + \sqrt{-1} \sum_{j=s+1}^n dz_j \wedge d \ov{z}_j.
\end{equation}
Here the equivalence means two metrics are mutually bounded.
\end{definition}

\begin{proof} [The proof of Theorem \ref{main-finite}]
We close this section by providing the proof of Theorem \ref{main-finite}, demonstrating the logarithmic vanishing theorem in the case where the number of irreducible components of the divisor is finite. The proof is similar to that in \cite{Norimatsu78}.
%More specifically, we deduce it from the inductions on the number of irreducible components. Therefore, this method is invalid for the case when the divisor has infinitely many irreducible components.
We consider the weakly pseudoconvex K\"ahler manifold $X$, and the divisor $D$ of normal crossing type, written as $D = D_1 + \cdots + D_N$. We define for any ordered multi-index $I =(i_1, \ldots, i_q)$ in $(1, \ldots, N)$
$$
D_I = D_{i_1} \cap \cdots \cap D_{i_q}
$$
and
$$
D^{[q]} = \coprod_{|I|=q} D_I, \quad~~D^{[0]} = X.
$$
Here the sign $\coprod$ means the disjoint union. It's important to note that in finite case all $D^{[q]}$ are manifolds (not necessarily connected), moreover, they are weakly pseudoconvex K\"ahler manifolds as well.
Following Deligne, we define the weight filtration $W_{\bullet}$ of $\Omega^p_X(\log D)$ by $W^p_k := W_k \Omega^p_X(\log D) := \Omega^k_X(\log D)\wedge \Omega^{p-k}_X$. We now define the $k$-residue $\Res_k$ of holomorphic logarithmic forms in $W^p_k$, for any $p$ logarithmic form $\alpha\in \Gamma(X, W_k\Omega_X^p(\log D))$, write
$$\alpha = \sum \alpha_I \wedge (\frac{dz}{z})^I,$$
where $I=(i_1, \ldots, i_k)$, and $(\frac{dz}{z})^I = \frac{dz_{i_1}}{z_{i_1}}\wedge\cdots\wedge\frac{dz_{i_k}}{z_{i_k}}$.
We set $\Res_k \alpha \in \Gamma(D^{[k]}, \Omega_{D^{[k]}}^{p-k})$ as
$$
(\Res_k \alpha)|_{D_I} := (\alpha_I)|_{D_I}.
$$

By the direct computation, we know $\Res_k \beta =0$ if $\beta\in W_{k-1}$. Moreover one has the next exact sequence ($p\geq k$):
\begin{align*}
0 \longrightarrow W^p_{k-1} \longrightarrow W^p_k \stackrel{\Res_k}\longrightarrow \Omega^{p-k}_{D^{[k]}} \longrightarrow 0.
\end{align*}
We can tensor the exact sequence with a line bundle $F$ and still keep the exactness, then obtain the long exact sequence of cohomology groups:
\begin{align*}
\cdots \longrightarrow H^q(X, W^p_{k-1}\otimes F) \longrightarrow H^q(X, W^p_k \otimes F) \longrightarrow H^q(D^{[k]}, \Omega^{p-k}_{D^{[k]}}\otimes F) \longrightarrow \cdots.
\end{align*}

Note that $F|_{D^{[k]}}$ is positive line bundle if $F$ is positive on $X$. According to the classic Akizuki--Nakano vanishing theorem on weakly pseudoconvex K\"ahler manifold, see for example \cite{Nakano73, Nakano74}, we know $H^q(D^{[k]}, \Omega^{p-k}_{D^{[k]}}\otimes F)=0$ for $p+q \geq n+1$ (note that $\dim_{\mathbb{C}}D^{[k]} = n-k$). Hence the natural morphism $H^q(X, W^p_{k-1}\otimes F) \rightarrow H^q(X, W^p_k \otimes F)$ is surjective for $p+q\geq n+1$.
We initiate the process with $k=p$, and then by successively replacing $k$ with $k-1$ and continuing this iteration, we obtain a surjective morphism $H^q(X, \Omega^p_X \otimes F) = H^q(X, W^p_0 \otimes F) \twoheadrightarrow H^q(X, W^p_p \otimes F) =H^q(X, \Omega^p_X (\log D)\otimes F)$. Utilizing the Akizuki--Nakano vanishing theorem, we deduce that $H^q(X, \Omega^p \otimes F)=0$. Consequently, for any $p+q \geq n+1$, we have $H^q(X, \Omega^p_X (\log D)\otimes F)=0$.
\end{proof}

\section{Analytic proof of the local vanishing theorem}
Consider $(X, \omega, \Phi)$ as a weakly pseudoconvex K\"ahler manifold with a smooth psh exhaustion function $\Phi$, as detailed in Remark \ref{sublevel}. For any positive real number $c$, the sublevel manifold is denoted as $X_c = \{x\in X: \Phi(x) < c \}$. A direct corollary of Theorem \ref{main-finite} is the following vanishing theorems on each sublevel manifold $X_c$, due to the finiteness of components of $D$ in each $X_c$.

\begin{theorem} \label{localvan-1}
Let $(F, h^F)$ be a positive holomorphic line bundle on an $n$-dimensional weakly pseudoconvex K\"ahler manifold $X$. For each real number $c$ and on the corresponding sublevel manifold $X_c$, we have the vanishing of cohomology groups,
$$
H^{q}(X_c, \Omega_X^{p}(\log D)\otimes F) = 0 \quad \text{for any} \quad p+q \geq n+1.
$$
\end{theorem}

In this section, we employ the $L^2$ method to establish the theorem mentioned above, specifically Theorem \ref{localvan}. We believe it is worthwhile to incorporate $L^2$ techniques in this type of vanishing problem. We begin by constructing a Poincar\'e type metric on $ Y_c:= X_c - D$.

%\begin{thm}
%Let $(X, \omega)$ be a K\"ahler manifold, here $\omega$ does not %necessary to be complete but the manifold $X$ have one complete %metric. Let $(F, h)$ be a Hermitian vector bundle of rank $r$ over %$X$, and assume that the curvature operator $A = %A_{F,h,\omega}^{p,q} = [\sqrt{-1}\Theta_{F,h}, \Lambda_{\omega}]$ %is semi-positive definite everywhere on $\Lambda^{p,q} T^{\ast}_X %\otimes F, q\geq 1$. Then for any form $g \in L^2(X, \Lambda^{p,q} %T^{\ast}_X \otimes F)$ satisfying $\dbar g =0$ and $\int_X \langle %A^{-1}g, g\rangle dV_{\omega} < +\infty$, there exists $f\in %L^2(X, \Lambda^{p,q-1} T^{\ast}_X \otimes F)$ such that $\dbar f %=g$ and
%$$
%\int_X |f|^2 dV_{\omega} \leq \int_X \langle  A^{-1}g, g\rangle %dV_{\omega}.
%$$
%
%\end{thm}

\begin{definition}[Incomplete Poincar\'e type K\"ahler metric on $X_c$] \label{pointype}
Let $D= \sum_i D_i$ be a simple normal crossing divisor of $X$, and $\sigma_i$ be the defining section of $D_i$. Fix any smooth Hermitian metrics $\| \cdot\|_i$ on $\mathcal{O}(D_i)$ such that $\|\sigma_i \|_i < 1$ on $X$ and $\|\sigma_i\|^2_i = \frac{1}{e}$ on a little bit away from $D_i$. Similar to \cite{Zucker79}, we set $\omega_{c,p} := (k_c \omega_c - \frac{1}{2} \xu \sum_i \pa \dbar \log \log^2 \|\sigma_i \|^2_i)$ for large positive integer $k_c$ which depends on $X_c$. By our assumption, the definition of $\omega_{c,p}$ makes sense since around any point on $X$, there are only finite non-zero terms in the sum $\frac{1}{2} \xu \sum_i \pa \dbar \log \log^2 \|\sigma_i \|^2_i$. In coordinates patch $W$ where $D_i$ is defined by $z_i=0$, and $\| \sigma_i \|^2_i = |z_i|^2 e^u$ for some function $u$ that is smooth on $W$. Then
$$
-\frac{1}{2} \pa \dbar \log \log^2 \|\sigma_i \|^2_i = \frac{1}{(\log|z_i|^2 + u)^2}(\frac{dz_i}{z_i}+\pa u) \wedge (\frac{d\ov{z}_i}{\ov{z}_i} + \dbar u) - \frac{1}{\log|z_i|^2 + u}\pa\dbar u.
$$
It is clear that $\omega_{c,p}$ is positive $(1,1)$-form on $Y_c$ and of Poincar\'e type along $D$ provided $k_c$ is sufficiently large as in Definition \ref{Poincare type metric}. But it is obvious that $\omega_{c,p}$ is not complete along the boundary of $Y_c$.
\end{definition}

Now, we will follow Huang--Liu--Wan--Yang's approach in \cite{HLWY16} to obtain the $L^2$ resolution.

%\begin{definition} \label{finesheaf}
%Let $(X_c, \omega_c)$ be the fixed sublevel manifold, let $h^F_{Y_c}$ be the Hermitian metric on $F|_{Y_c}$.
%The sheaf $\Xi^{p,q}_{(2)}(X_c, F, \omega_{c,p}, h^F_{Y_c})$ over $X_c$ is defined as follows. On any open subset $U$ of $X_c$, the section space $\Gamma(U, \Xi^{p,q}_{(2)}(X_c, F, \omega_{c,p}, h^F_{Y_c}))$ over $U$ consists of $F$-valued $(p,q)$-forms $u$ with measurable coefficients such that the $L^2$ norms of both $u$ and $\dbar u$ are integrable on any compact subset $K$ of $U$. Here the integrability means that both $|u|^2_{\omega_{c,p}\otimes h^F_{Y_c}}$ and $|\dbar u|^2_{\omega_{c,p}\otimes h^F_{Y_c}}$ are integrable on $K\backslash D$.

%Recall that a sheaf $\mathcal{F}$ on $X$ is called a fine sheaf if, for any locally finite open covering $\{U_i\}$ of $X$, there is a family of homomorphisms $\{ f_i\}$, and each $f_i: \mathcal{F} \rightarrow \mathcal{F}$, such that
%\begin{enumerate}
 %   \item supp $f_i \subset U_i$,
   % \item $\sum_i f_i =1$, i.e., $\sum_i f_i(s) =s$ for any section $s$ of $\mathcal{F}$.
%\end{enumerate}
%If the metric $\omega_{c,p}$ is of Poincar\'e type as in Definition \ref{pointype}, then it is complete along the divisor $D$ and is of finite volume, see for example \cite[Proposition 3.4]{Zucker79}. As a consequence, the sheaf $\Xi^{p,q}_{(2)}(X_c, F, \omega_{c,p}, h^F_{Y_c})$ would be a fine sheaf (on $X_c$).
%\end{definition}

\begin{definition} \label{finesheaf}
Let $(X_c, \omega_c)$ denote a sublevel manifold, and let $h^F_{Y_c}$ represent the restriction Hermitian metric of $h^F$ on $F_{|{Y_c}}$.
The sheaf $\Xi^{p,q}_{(2)}(X_c, F, \omega_{c,p}, h^F_{Y_c})$ over $X_c$ is defined as follows: On any open subset $U$ of $X_c$, its section space over $U$ consists of $F$-valued $(p,q)$-forms $u$ with measurable coefficients such that the $L^2$ norms of both $u$ and $\dbar u$ are integrable on any compact subset $K$ of $U$. Here, integrability implies that both $|u|^2_{\omega_{c,p}\otimes h^F_{Y_c}}$ and $|\dbar u|^2_{\omega_{c,p}\otimes h^F_{Y_c}}$ are integrable on $K\backslash D$.

Recall that a sheaf $\mathcal{F}$ on complex $X$ is termed a fine sheaf if it admits a partition of the unit, i.e., for any locally finite open covering $\{U_i\}$ of $X$, there exists a family of smooth functions $\{f_i\}$ satisfying:
\begin{enumerate}
\item supp $f_i \subset U_i$,
\item $\sum_i f_i = 1$ on $X$.
\end{enumerate}
In the context of our case $(X_c, \omega_c)$, if the metric $\omega_{c,p}$ conforms to the Poincar\'e type as described in Definition \ref{pointype}, it is complete along the divisor $D$ and possesses finite volume, as demonstrated, for instance, in \cite[Proposition 3.4]{Zucker79}. Consequently, the sheaf $\Xi^{p,q}_{(2)}(X_c, F, \omega_{c,p}, h^F_{Y_c})$ qualifies as a fine sheaf on $X_c$. Note that it may not be a fine sheaf on $Y_c=X_c\backslash D$.
\end{definition}

\begin{theorem} [An $L^2$-type Dolbeault isomorphism] \cite[Theorem 3.1]{HLWY16}\label{L2iso}
Let $(X, \omega)$ be a weakly pseudoconvex K\"ahler manifold of dimension $n$, and let $D=\sum_{i}D_i$ be a simple normal crossing divisor in $X$. For a fixed real number $c$, let $X_c$ denote the associated sublevel manifold. We can construct a K\"ahler metric $\omega_{c,p}$ on $Y_c$ which is of Poincar\'e type along $D$ as described in Definition \ref{pointype}. For a line bundle $(F,h^F)$, there exists a Hermitian metric $h^F_{Y_c,\alpha_c}$ on $F_{|{Y_c}}$ such that the sheaf $\Omega^p_{X}(\log D)\otimes F$ over $X_c$ admits a fine resolution given by the $L^2$ Dolbeault complex $(\Xi^{p,\ast}_{(2)}(X_c, F, \omega_{c,p}, h^F_{Y_c,\alpha_c}), \dbar)$, where $\alpha_c$ is a (to be determined) large positive constant dependent on $X_c$. This implies the existence of an exact sequence of sheaves over $X_c$:
\begin{equation} \label{complex-3.4}
0 \rightarrow \Omega^p_{X_c}(\log D_c)\otimes F \rightarrow \Xi^{p,0}_{(2)}(X_c, F, \omega_{c,p}, h^F_{Y_c,\alpha_c}) \rightarrow \Xi^{p,1}_{(2)}(X_c, F, \omega_{c,p}, h^F_{Y_c,\alpha_c}) \rightarrow \ldots
\end{equation}
such that $\Xi^{p,q}_{(2)}(X_c, F, \omega_{c,p}, h^F_{Y_c,\alpha_c})$ is a fine sheaf for each $0\leq p,q \leq n$. In particular, by Dolbeault's isomorphism, we have:
\begin{equation}
H^q(X_c,\Omega^p_{X_c}(\log D_c)\otimes F) \simeq H^{p,q}_{(2)}(Y_c,F,\omega_{c,p},h^F_{Y_c,\alpha_c}).
\end{equation}

Here, $H^{p,\ast}_{(2)}(Y_c, F,\omega_{c,p},h^F_{Y_c,\alpha_c})$ denotes the cohomology associated with the global sections of the complex $\Xi^{p,\ast}_{(2)}:= (\Xi^{p,\ast}_{(2)}(X_c, F, \omega_{c,p}, h^F_{Y_c,\alpha_c}), \dbar)$ on $X_c$, this is to say, the cohomology of the following complex:
\begin{equation*}
 0\rightarrow \Gamma(X_c, \Xi^{p,0}_{(2)}) \rightarrow \Gamma(X_c, \Xi^{p,1}_{(2)}) \rightarrow \ldots \rightarrow \Gamma(X_c, \Xi^{p,n}_{(2)})\rightarrow 0.
\end{equation*}

\end{theorem}

\begin{proof}
The proof comes from \cite[Theorem 3.1]{HLWY16}, We provide the construction of the metric  $h^F_{Y_c,\alpha_c}$ on $F_{|{Y_c}}$ that renders the complex \eqref{complex-3.4} exact.
To begin, for any fixed constants $\tau_i \in (0,1]$, we establish a smooth Hermitian metric on $F|_{Y_c}$ as follows:
$$
h^F_{Y_c, \alpha_c} := \prod_{i\in \Lambda_c} \|\sigma_i \|^{2\tau_i}_i (\log^2 \|\sigma_i \|^2_i)^{\frac{\alpha_c}{2}} h^F,
$$
where $\alpha_c$ is a large positive constant to be determined. Under Definition \ref{finesheaf}, it suffices to verify the exactness of the complex \eqref{complex-3.4}. For the convenience of the reader, we present the proof of exactness solely at $q=0$ because I think this part can help the reader understand the $L^2$ resolution of logarithmic forms here. Those interested in the complete proof may refer to \cite{HLWY16}.

Let $(W;z_1,\ldots,z_n)$ be a local coordinate chart of $X_c$ along $D$, where $D=\{z_1\cdots z_t=0\}$. Let $e$ be a trivialization section of $F$ on $W$ such that
$\frac{1}{2}\leq|e(z)|_{h^F}\leq1$ over $W$. Let's define the logarithmic coordinate as follows:
\begin{align*}
  \zeta_j&=\frac{1}{z_j}dz_j, \quad \text{for}\quad 1\leq j\leq t, \\
   \zeta_j&=dz_j, \quad \text{for}\quad t+1\leq j\leq n.
\end{align*}
Suppose $\sigma$ is a holomorphic section of $\Xi^{p,0}_{(2)}(X_c, F, \omega_{c,p},  h^F_{Y_c,\alpha_c})$. Then, we can express $\sigma$ as:
\begin{equation}
  \sigma(z)=\sum_{|I|=p}\sigma_I(z)\zeta_{i_1}\wedge\cdot\cdot\cdot\wedge\zeta_{i_p}\otimes e\nonumber
\end{equation}
where $I=(i_1,\ldots,i_p)$ is a multi-index with
$i_1<\cdot\cdot\cdot<i_p$ and $\sigma_I(z)$ is a holomorphic
function on $W_{1/2}^{\ast}$. By definition, on
$W_r^{*} := \Delta_{r}^{*t}\times\Delta_{r}^{n-t}\subset
W_{1/2}^{*}$ for any $0<r<1/2$, we see that $\sigma$ is integrable in the $L^2$ sense with respect to $\omega_{c,p}$ and $h^F_{Y_c,\alpha_c}$.

Note that the Hermitian metric
$h^{F}_{Y_c, \alpha_c}|_{W_{1/2}^{*}}$ is equivalent to the following
Hermitian metric
\begin{equation*}
  h^{F}_{\alpha_c}(W_{1/2}^{*})=\prod_{i=1}^t|z_i|^{2\tau_i}(\log^2|z_i|^{2})^{\frac{\alpha_c}{2}}h^F.
\end{equation*}
If we denote $\{i_1,\ldots,i_p\}\cap\{1,\ldots,t\}=\{i_{p1},\ldots,i_{pb}\}$,
then one can write
\begin{align}
  \|\sigma\|_{L^2(W_r^{*})}^2=\sum_{|I|=p}\int_{W_r^{*}} |e|_{h^L}^2
  \left(|\sigma_I(z)|^2 \prod_{\nu=1}^b\log^2|z_{i_{p\nu}}|^2
  \prod_{i=1}^t|z_i|^{2\tau_i}(\log^2|z_i|^{2})^{\frac{\alpha_c}{2}}
  \right)\omega_{c,p}^{n}.\label{integration}
\end{align}
Note that the Poincar\'e metric $\omega_{c,p}$ is equivalent to the following
\begin{equation*}
    \omega_{c,p} = \sqrt{-1} \sum_{j=1}^t \frac{dz_j \wedge d \ov{z}_j}{|z_j|^2\cdot \log^2|z_j|^2} + \sqrt{-1} \sum_{j=t+1}^n dz_j \wedge d \ov{z}_j.
\end{equation*}
Let's assume that the Laurent series representation of $\sigma_I(z)$ on $W_{1/2}^{*}$ is given by:
\begin{equation}
  \sigma_I(z)=\sum_{\beta=-\infty}^{\infty}\sigma_{I\beta}(z_{t+1},\ldots,z_n)z_1^{\beta_1}\cdot\cdot\cdot z_t^{\beta_t}, \beta=(\beta_1,\ldots,\beta_t)\in \mathbb Z^t\nonumber
\end{equation}
where $\sigma_{I\beta}(z_{t+1},\ldots,z_n)$ is a holomorphic function
on $\Delta_{1/2}^{n-t}$. 

Using polar coordinates and Fubini's theorem, we ascertain that $\sigma$ is $L^2$ integrable on $W_r^{\ast}$ if and only if each $\beta_i > -\tau_i$ along $D_i$. Since $\tau_i \in (0,1]$, we deduce that $\beta_i \geq 0$, and thus, $\sigma_I(z)$ has a removable singularity. Consequently, both $\sigma$ and $\nabla \sigma$ possess only logarithmic poles, and $\sigma$ becomes a section of $\Omega^{p}(\log D)\otimes F$ on $W$. Conversely, if we select $\sigma$ as a holomorphic section of $\Omega^{p}(\log D)\otimes F$ on $W$, it's straightforward to verify, using formula (\ref{integration}), that $\sigma$ is $L^2$ integrable on $W_r^{\ast}$ for any $0 < r < \frac{1}{2}$. Thus, we have demonstrated the exactness of the complex at $\Xi^{p,0}_{(2)}(X_c, F, \omega_{c,p}, h^F_{Y_c,\alpha_c})$ for any $\alpha_c > 0$.
\end{proof}

Even though $\omega_{c,p}$ is not complete, $Y_c$ admits a complete K\"ahler metric. Indeed, let $\tilde{\omega}=\hat{\omega}+\omega_{c,p}$, here $\hat{\omega}$ is complete along the boundary of $X_c$ like in the Proposition \ref{hatom}. We know that $\tilde{\omega}$ is complete on $Y_c$.
Hence we can still solve the certain $\dbar$-equation on $Y_c$ thanks to Theorem \ref{dbareq} when $p=n$. We slightly modify Huang--Liu--Wan--Yang's approach in \cite{HLWY16} to get the local vanishing.

\begin{theorem} \label{localvan}
Let $(F, h^F)$ be a positive holomorphic line bundle on an $n$-dimensional weakly pseudoconvex K\"ahler manifold $X$. For each positive real number $c$ and on the corresponding sublevel manifold $X_c$, we observe the vanishing of cohomology groups,
$$
H^{q}(X_c, \Omega_X^{p}(\log D)\otimes F) = 0 \quad \text{for any} \quad p+q \geq n+1.
$$
\end{theorem}

%It would be interesting to prove the above local vanishing theorem for general $\Omega_X^p\otimes \mathcal{O}(D)\otimes F$ (i.e., Theorem \ref{localvan-1}) by the $ L^2$ method.
\begin{proof}
Let $\omega_c$ be a fixed K\"ahler metric on $X_c$. We denote $\{\lambda^j_{\omega_c}(h^F)\}_{j=1}^n$ as the increasing sequence of eigenvalues of $\xu\Theta(F,h^F)$ with respect to $\omega_c$. As $X_c$ is relatively compact in $X$, there exists a positive constant $c_0$ such that the smallest eigenvalue $\lambda^1_{\omega_c}(h^F) \geq c_0$ everywhere on $X_c$. Now, we proceed to construct a new metric on $F|_{Y_c}$ as follows:
$$
 h_{\alpha,\epsilon,\tau} := \prod_{i\in \Lambda_c} \|\sigma_i \|^{2\tau_i}_i (\log^2(\epsilon \|\sigma_i \|^2_i))^{\frac{\alpha}{2}}\cdot h^F \cdot e^{-\chi\circ\phi}.
$$
Here, the constant $\alpha >0$ is chosen to be sufficiently large to satisfy the condition in Theorem \ref{L2iso}, while the constants $\tau_i$ and $\epsilon$, both in the interval $(0,1]$, are to be determined later and depend on $X_c$. Note that here $\chi\circ\phi$ stems from Proposition \ref{hatom}, which ensures that the metric we will soon construct becomes complete. On $Y_c$, straightforward computation yields:
\begin{align} \label{curvature1}
    \xu \Theta(F, h_{\alpha,\epsilon,\tau}) =  &\xu \Theta(F,h^F) + \sum_{i\in \Lambda_c} \tau_i c_1(D_i) + \xu \pa\dbar(\chi\circ\phi) \\
  \nonumber  &+ \sum_{i\in \Lambda_c} \frac{\alpha c_1(D_i)}{\log(\epsilon\|\sigma_i\|^2_i)} + \xu \sum_{i\in \Lambda_c} \frac{\alpha \pa \log\|\sigma_i\|^2_i \wedge \dbar \log\|\sigma_i\|^2_i}{(\log(\epsilon\|\sigma_i\|^2_i))^2}.
\end{align}
The third term and the final term on the right-hand side are a semi-positive $(1,1)$-form. Let's set $\delta = \frac{c_0}{8n-1}$. Then, we can select $\tau_i$ and $\epsilon$ to be sufficiently small such that
\begin{equation} \label{curvature2}
    -\frac{\delta}{2}\omega_c \leq \sum_{i\in \Lambda_c} \tau_i c_1(D_i) \leq \frac{\delta}{2}\omega_c, \quad \quad -\frac{\delta}{2}\omega_c  \leq  \sum_{i\in \Lambda_c} \frac{\alpha c_1(D_i)}{\log(\epsilon\|\sigma_i\|^2_i)}    \leq   \frac{\delta}{2}\omega_c,
    \end{equation}
hold on $Y_c$. The constants $\tau_i$ and $\epsilon$ are thus fixed. We set
\begin{equation*}
    \omega_{Y_c} := \xu\Theta(F, h_{\alpha,\epsilon,\tau}) + 2\delta \omega_c.
\end{equation*}
From formula \eqref{curvature1}, we observe that $\omega_{Y_c}$ is a Poincar\'e type K\"ahler form on $Y_c$ along $D$ (Note that $\chi\circ\phi$ is smooth along $D$, hence innocuous). Furthermore, it's clear from formula \eqref{curvature2} that we have:
\begin{equation} \label{curvature3}
     \xu \Theta(F, h_{\alpha,\epsilon,\tau}) \geq \xu \Theta(F,h^F) -\delta \omega_c
\end{equation}
on $Y_{c}$. Since $ \xu \Theta(F,h^F)$ is a positive $(1,1)$-form, we know on $Y_c$
\begin{equation} \label{metric-rel}
    \omega_{Y_c} = \xu\Theta(F, h_{\alpha,\epsilon,\tau}) + 2\delta \omega_c \geq \delta \omega_c,
\end{equation}
which means $\omega_{Y_c}$ is a positive $(1,1)$-form and thus a metric. Moreover, according to \eqref{curvature1} and Proposition \ref{hatom}, this metric is complete on $Y_c$.
  
In a local chart of $Y_c$, let's assume that $\omega_c = \xu \sum_{i=1}^{n}\eta_i \wedge \ov{\eta}_i$ and
\begin{align*}
\xu \Theta(F, h_{\alpha,\epsilon,\tau}) &= \xu \sum_{i=1}^{n} \lambda_{\omega_c}^i(h_{\alpha,\epsilon,\tau}) \eta_i \wedge \ov{\eta}_i \\
 &= \xu \sum_{i=1}^{n} \frac{\lambda_{\omega_c}^i(h_{\alpha,\epsilon,\tau})}{\lambda_{\omega_c}^i(h_{\alpha,\epsilon,\tau})+2\delta} \eta_i' \wedge \ov{\eta}_i'
\end{align*}
where
$$
\eta_i' = \eta_i \sqrt{\lambda_{\omega_c}^i(h_{\alpha,\epsilon,\tau})+2\delta}.
$$
Note that $\omega_{Y_c} = \xu \sum_{i=1}^{n}\eta_i' \wedge \ov{\eta}_i'$, as per formula \eqref{metric-rel}. Consequently, the $i$-th eigenvalues of $\xu \Theta(F, h_{\alpha,\epsilon,\tau})$ with respect to $\omega_{Y_c}$ is
$$
\gamma_i := \frac{\lambda_{\omega_c}^i(h_{\alpha,\epsilon,\tau})}{\lambda_{\omega_c}^i(h_{\alpha,\epsilon,\tau})+2\delta} <1.
$$
On the other hand, due to \eqref{curvature3} one has
$$
\lambda_{\omega_c}^i(h_{\alpha,\epsilon,\tau}) \geq c_0 - \delta.
$$
It implies
$$
 \gamma_i = \frac{\lambda_{\omega_c}^i(h_{\alpha,\epsilon,\tau})}{\lambda_{\omega_c}^i(h_{\alpha,\epsilon,\tau})+2\delta} \geq \frac{c_0-\delta}{c_0+\delta} =\frac{c_0-\frac{c_0}{8n-1}}{c_0+\frac{c_0}{8n-1}} = 1-\frac{1}{4n}.
$$
For any section $u\in \Gamma(Y_c, \wedge^{p,q}T^{\ast}Y_c\otimes F)$, we get
\begin{align} \label{eigenineq}
    \left \langle  [\xu \Theta(F, h_{\alpha,\epsilon,\tau}), \Lambda_{\omega_{Y_c}}]u, u\right \rangle &\geq \big( \sum_{i=1}^q \gamma_i - \sum_{j=p+1}^n \gamma_j\big)|u|^2 \\
   \nonumber &\geq \big(q(1-\frac{1}{4n}) - (n-p)   \big)|u|^2 \\
   \nonumber &\geq \frac{1}{2}|u|^2.
\end{align}
Readers can refer to \cite[Formula 4.10]{Dem12a} for the first inequality. The last inequality holds when $p+q \geq n+1$. Because $\omega_{Y_c}$ is complete, the above inequality \eqref{eigenineq} remains valid for $L^2$-forms through approximation. 
Hence, by Theorem \ref{dbareq-2}, we achieve the vanishing of $H^{p,q}_{(2)}(Y_c,F,\omega_{Y_c},h_{\alpha,\epsilon,\tau})$ fro $p+q \geq n+1$.

Since we know that $\omega_{Y_c}$ is a Poincar\'e type metric along $D$ on $X_c$, according to Theorem \ref{L2iso}, we have:
$$
H^q(X_c,\Omega^p(\log D)\otimes F) \simeq H^{p,q}_{(2)}(Y_c,F,\omega_{Y_c},h_{\alpha,\epsilon,\tau}).
$$
Consequently, we obtain the desired vanishing theorem.

%When $p=n$, the inequality \eqref{eigenineq} and Theorem \ref{dbareq} demonstrate the vanishing of 
%$$H^{n,q}_{(2)}(Y_c,F,\omega_{Y_c},h_{\alpha,\epsilon,\tau})=0$$
%for any $q\geq 1$. Consequently, we obtain the desired vanishing theorem. Note that Theorem $2.8$ may fail when $p\neq n$.
\end{proof}

\section{Global vanishing theorem}
Studying the cohomology groups on weakly pseudoconvex manifolds offers the advantage of investigating the corresponding higher direct images. Consider a proper surjective morphism $f: X\rightarrow S$ from a K\"ahler manifold $X$ to a reduced and irreducible complex space $S$. Let $W \subset S$ be any Stein open subset, and put $V = f^{-1}(W)$. Then, $V$ is a weakly pseudoconvex K\"ahler manifold. Indeed, for a psh exhaustion function $\phi$ on $W$, $f^{\ast}\phi$ would be a psh exhaustion function on $V$.
Let $\mathcal{F}$ be a coherent sheaf on $V$. According to \cite[Lemma II.1]{Pr71}, the map $f^{\ast} : H^q(V, \mathcal{F}) \rightarrow H^0(W, R^qf_{\ast}\mathcal{F})$ is an isomorphism of topological vector spaces for every $q \geq 0$.
As a direct corollary of Theorem \ref{localvan-1}, we derive:
\begin{corollary} \label{vanofdirect}
Let $f: X \rightarrow S$ be a proper holomorphic morphism from a K\"ahler manifold $X$ onto the reduced and irreducible complex space $S$. Let $D$ be a simple normal crossing divisor for which $f|_D$ is proper. And let $F$ be a positive holomorphic line bundle on $X$, then
$$
 R^qf_{\ast}(\Omega_X^p(\log D)\otimes F) =0 \quad \text{for any} \quad p+q \geq n+1.
$$
\end{corollary}
\begin{proof}
Given any smooth Stein open subset $W \subset S$ that is relatively compact, we have
    $$H^0(W, R^qf_{\ast}(\Omega_X^p(\log D)\otimes F)) \simeq H^q(f^{-1}(W), \Omega_X^p(\log D)\otimes F).$$
    We know $f^{-1}(W)$ is weakly pseudoconvex and relatively compact in $X$. Following Theorem \ref{localvan-1}, we know $H^q(f^{-1}(W), \Omega_X^p(\log D)\otimes F)=0$ for any $p+q \geq n+1$.
\end{proof}

On holomorphically convex K\"ahler manifolds, the global vanishing can be deduced from the local vanishing.

\begin{corollary} \label{vanofholo}
Let $X$ be a holomorphically convex K\"ahler manifold and $F$ be a positive line bundle on $X$. Let $D$ be a simple normal crossing divisor on $X$, here the number of irreducible components of the divisor may be infinite. We have
$$
H^q(X, \Omega_X^p(\log D)\otimes F) = 0 \quad \text{for any} \quad p+q \geq n+1.
$$
\end{corollary}

\begin{proof}
Consider the Remmert reduction $f: X \rightarrow S$ as described in Remark \ref{remmert}. For the locally free sheaf $\mathcal{G} := \Omega_X^p(\log D)\otimes F$, we have the Leray spectral sequence
$$
H^k(S, R^qf_{\ast}\mathcal{G}) \Rightarrow H^{k+q}(X, \mathcal{G}).
$$
Since $S$ is a normal Stein space, thanks to Cartan's theorem B, we have the following vanishing
$$
H^k(S, R^qf_{\ast}(\Omega_X^p(\log D)\otimes F)) = 0
$$
for any $p, q\geq 0$ and $k\geq 1$.
On the other hand, as Corollary \ref{vanofdirect} states, the local vanishing implies the vanishing of the high direct image sheaf $R^{q}f_{\ast}(\Omega_X^p(\log D)\otimes F)=0$ for any $p+q\geq n+1$. These two facts together yield the vanishing $H^q(X, \Omega_X^p(\log D)\otimes F) = 0$ for any $p+q\geq n+1$.
\end{proof}

Now, focusing on weakly pseudoconvex K\"ahler manifolds, we begin to prove Theorem \ref{main2}. Firstly, we have
\begin{theorem} \label{q2}
Let $X$ be a weakly pseudoconvex K\"ahler manifold and $F$ be a positive line bundle on $X$. Let $D$ be a simple normal crossing divisor on $X$. We have
$$
H^q(X, \Omega_X^n(\log D)\otimes F) = H^q(X, K_X\otimes \mathcal{O}(D)\otimes F) =0,
$$
for any $q \geq 2$.
\end{theorem}
At present, we can not prove the global vanishing of $H^1(X, K_X\otimes \mathcal{O}(D)\otimes F)$. We will deal with it later.

\begin{definition} \label{leray}
    First, let us recall the definition of Leray covering, it is a cover of a topological space that allows for easy calculation of its cohomology. Let $\mathfrak{U} = \{U_i\}$ be an open cover of the complex manifold $X$, and $\mathcal{F}$ a sheaf on $X$. We say $\mathfrak{U}$ is a Leray conver with respect to $\mathcal{F}$ if, for every nonempty finite set $\{i_1, \ldots, i_n\}$ of indices, and for all $k>0$, we have $H^k(U_{i_1}\cap \cdots \cap U_{i_n}, \mathcal{F}) =0$. We denote the $q$-cochain, $q$-cocycle, and $q$-coboundary of sheaf $\mathcal{F}$ with respect to covering $\mathfrak{U}$ by $C^q(\mathfrak{U}, \mathcal{F})$, $Z^q(\mathfrak{U}, \mathcal{F})$ and $B^q(\mathfrak{U}, \mathcal{F})$ respectively.
\end{definition}

\begin{proof} [The proof of Theorem \ref{q2}]
According to Sard theorem, we can choose a sequence $\{c_v \}_{v=0,1,\cdots}$ of real numbers such that
\begin{enumerate}
    \item $c_v < c_{v+1}$ and $\lim_{v \to \infty}c_v = +\infty$;
    \item the bondary $\pa X_v$ of $X_v=\{x\in X : \Phi(x)< c_v\}$ is smooth for any $v$.
\end{enumerate}
Consider $X_{0} \subset X_{1} \subset \cdots \subset X_{k} \subset \cdots$ as an exhaustion sequence of $X$. Thus, $\mathfrak{X} = \{X_v\}_{v\geq0}$ constitutes a covering of $X$. For any $v$, let $\mathfrak{X}_v = \{X_k\}_{k\leq v}$, where $k$ ranges over non-negative integers. Then, $\mathfrak{X}_v$ forms a covering of $X_v$. By Theorem \ref{localvan-1}, the cohomology of $\Omega_X^n(\log D)\otimes F$ vanishes on each sublevel manifold $X_v$. Therefore, this covering $\mathfrak{X}$ (and similarly $\mathfrak{X}_v$) constitutes the Leray covering with respect to the sheaf $\Omega_X^n(\log D)\otimes F$ on $X$ (and $X_v$). Therefore we have, for any $q\geq 1$ and $v\geq0$,
$$
H^q(X, \Omega^n(\log D)\otimes F) = \check{H}^q(\mathfrak{X}, \Omega^n(\log D)\otimes F)
$$
and
$$
H^q(X_v, \Omega^n(\log D)\otimes F) = \check{H}^q(\mathfrak{X}_v, \Omega^n(\log D)\otimes F) =0.
$$
The right cohomology groups are \v{C}ech cohomology groups.

For any $q \geq 1$, let's consider a $q$-cocycle $\sigma \in Z^q(\mathfrak{X}, \Omega^n(\log D)\otimes F)$, and define $\sigma_v$ as the restriction of $\sigma$ to $\mathfrak{X}_v$. Then we can conclude that $\sigma_v \in Z^q(\mathfrak{X}_v, \Omega^n(\log D)\otimes F)$. 
Indeed, if we express $\sigma$ as the set $\{\sigma_{i_0, i_1, \ldots, i_q}\}$ with $i_0 < i_1 < i_2 < \cdots < i_q$, where each $\sigma_{i_0, i_1, \ldots, i_q} \in H^0(X_{i_0}\cap X_{i_1}\cap \cdots \cap X_{i_q}, \Omega^n(\log D)\otimes F)$, note that $X_{i_0}\cap X_{i_1}\cap \cdots \cap X_{i_q} = X_{i_0}$ since $X_{i_0} \subset X_{i_1} \subset X_{i_2} \subset \cdots \subset X_{i_q}$. Now, let $\sigma_{v} := \{\sigma_{i_0,i_1,\ldots, i_q}\}_{i_q \leq v}$ be the subset of $\sigma$, and it's evident that we have the recurrence relation $\sigma_{v-1} =\sigma_v|_{\mathfrak{X}_{v-1}}$.

According to the local vanishing, there is a $(q-1)$-cochain $\alpha_v \in  C^{q-1}(\mathfrak{X}_v, \Omega^n(\log D)\otimes F)$ such that $\delta \alpha_v = \sigma_v$. The notation $\delta$ here is the so-called \v{C}ech differential, note that we have $(\delta \alpha)_v = \delta (\alpha_v)$, i.e., The differential commutes with the restriction. As an element of $ C^{q-1}(\mathfrak{X}_{v-1}, \Omega^n(\log D)\otimes F)$, we have $\delta \alpha_v = \delta \alpha_{v-1}$, here $\delta \alpha_v$ should be its restriction, denoted as $(\delta \alpha_v)_{\mathfrak{X}_{v-1}}$, but we omit the indices for simplicity. Hence $\alpha_v - \alpha_{v-1} \in Z^{q-1}(\mathfrak{X}_{v-1}, \Omega^n(\log D)\otimes F)$.

By assumption $q \geq 2$, so there is a $(q-2)$-cochain $\beta_{v-1} \in C^{q-2}(\mathfrak{X}_{v-1}, \Omega^n(\log D)\otimes F)$ such that $\delta \beta_{v-1} = \alpha_v - \alpha_{v-1}$ in $C^{q-1}(\mathfrak{X}_{v-1}, \Omega^n(\log D)\otimes F)$. We now can define $\alpha' \in C^{q-1}(\mathfrak{X}, \Omega^n(\log D)\otimes F)$, the $(q-1)$-cochain on $\mathfrak{X}$ as follows (note that as an element in $C^{q-1}(\mathfrak{X}, \Omega^n(\log D)\otimes F)$, if we know its restriction on each $\mathfrak{X}_v$, then we possess complete information about it).

On each $\mathfrak{X}_v$,
\begin{align*}
\alpha' =& \alpha_v -\delta(\sum_{k<v}\beta_k)  \\
=& \alpha_v -\delta(\beta_{v-1} )
-\delta(\beta_{v-2} ) \cdots
-\delta(\beta_{1} ).
\end{align*}
Note that even if $\beta_{k}=\{\beta_{j_0,j_1,\ldots,j_{q-2}}\} \in C^{q-2}(\mathfrak{X}_{k}, \Omega^n(\log D)\otimes F)$ is merely the $(q-2)$-cochain on $\mathfrak{X}_{k}$, we can regard $\beta_{k}$ as an element of $C^{q-2}(\mathfrak{X}_{v}, \Omega^n(\log D)\otimes F)$ by adding some trivial sections.
Indeed, it is natural to set $\beta_k =\{\tilde{\beta}_{j_0,j_1,\ldots,j_{q-2}}\}$, here 
\begin{align*}
    &\tilde{\beta}_{j_0,j_1,\ldots,j_{q-2}} = \beta_{j_0,j_1,\ldots,j_{q-2}}, ~~~~\text{if}\quad j_0, j_1, \ldots, j_{q-2} \leq k,\\
    &\tilde{\beta}_{j_0,j_1,\ldots,j_{q-2}} = 0 ~~~~~\text{if else}.
\end{align*}

On $\mathfrak{X}_{v+1}$, similarly we have
\begin{align*}
\alpha' =& \alpha_{v+1} -\delta(\sum_{k<v+1}\beta_k)  \\
=& \alpha_{v+1} -\delta(\beta_{v})
-\delta(\beta_{v-1}) \cdots
-\delta(\beta_{1}).
\end{align*}
According to the definition of $\beta_v$, we have $\alpha_{v+1} -\delta(\beta_{v}) = \alpha_v$ on $\mathfrak{X}_v$ as the elements of $C^{q-1}(\mathfrak{X}_{v}, \Omega^n(\log D)\otimes F)$. It follows that $\alpha'$ is well defined. Finally, on each $\mathfrak{X}_v$,
$$
\delta \alpha' = \delta \alpha_v - \delta \delta(\sum_{k<v}\beta_k) = \delta \alpha_v = \sigma_v.
$$
Hence we have $\delta \alpha' = \sigma$ on $\mathfrak{X}$ as the elements of $C^{q}(\mathfrak{X}, \Omega^n(\log D)\otimes F)$. This yields the vanishing of cohomology groups.
\end{proof}

Now we give the proof of the vanishing of $H^1(X, K_X\otimes \mathcal{O}_X(D)\otimes F)$ on weakly pseudoconvex K\"ahler manifold. The key method is a Runge-type approximation used in \cite{Nakano70, Nakano73, Kazama73, Take81, TaOh81}.
For any real number pair $c_1 < c_2$, let $X_1 := \{x\in X : \Phi(x)< c_1 \}$ and $X_2 := \{x\in X : \Phi(x)< c_2 \}$. Set $Y_1 := X_1\backslash D$ and $Y_2 := X_2\backslash D$. As soon as we have the fixed pair $(X_1, X_2)$ and $(Y_1, Y_2)$. We select a smooth Hermitian metric $\{h_1 = e^{-\phi_1}\}$ on $\mathcal{O}(D)$.
We can construct the canonical singular metric $h_2$ on $\mathcal{O}(D)$ as follows: Let $g$ represent the natural section of the effective divisor $D$. Then, we have the global function $\phi := \log(|g|^2_{h_1})$. Consequently, we define $h_2 := h_1 e^{-\phi} = \frac{1}{|g|^2}$, which serves as a singular metric on $\mathcal{O}(D)$.
Locally, we can express $\{h_2 = e^{-\phi_2}\}$ with $\phi_2 = \sum_i \log|g_i|^2$, where $g_i$ denotes the generator of $D_i$. It's important to note that this sum remains finite within any local patch. Let $(F, h^F)$ be the fixed positive line bundle, we construct a new metric $h_{\delta}$ on the $F_D:= \mathcal{O}(D) \otimes F$,
\begin{equation} \label{con-metric}
h_{\delta} := h^F \prod_{i\in \Lambda} \big(\log^2(\|\sigma_i \|^2_i) \big)^{\frac{\kappa}{2}} e^{-\delta\phi_1} e^{-(1-\delta)\phi_2} \quad (0<\delta < 1).
\end{equation}
Here the constant $\kappa$ and $\delta$ are to be determined later. Also note that only finite many $\log (\|\sigma_i \|^2_i)\neq 1$ by Definition \ref{pointype} on any local coordinate patch.
According to \eqref{con-metric}, we know the associated multiplier ideal sheaf $\mathcal{I}(h_{\delta}) = \mathcal{O}_{X_2}$ on $X_2$ because $0< \delta <1$. Note that the logarithmic factor does not impact the local integrability.

On $Y_2$, $h_{\delta}$ is smooth and the curvature forms
$$
\xu \Theta_{F_D,h_{\delta}} = \xu \Theta_{F,h^F} + (-\kappa \xu \sum\pa\dbar \log(\log^2\|\sigma_i\|^2_i)) + \xu \delta \pa\dbar \phi_1.
$$
Set $L^{n,0}(Y_2, F_D, h_{\delta})$ be the set of $F_D$-valued $(n,0)$-form on $Y_2$ with a finite $L^2$ norms with respect to $h_{\delta}$, this norm is independent on K\"ahler metric since we are focusing on the $(n,0)$-forms.
We Set $\mathcal{A}^{n,0}(Y_2, F_D, h_{\delta}) := \ker \dbar \cap L^{n,0}(Y_2, F_D, h_{\delta})$ and moreover we have the following
$$
H^0(X_2, K_X\otimes F\otimes \mathcal{O}(D)) = H^0(X_2, K_X\otimes F\otimes \mathcal{O}(D)\otimes \mathcal{I}(h_{\delta})) \supseteq \mathcal{A}^{n,0}(Y_2, F_D, h_{\delta})
$$
because of the $L^2$ extension property of holomorphic $(n,0)$-forms. Indeed, since on any local coordinate neighborhood, the singular metric $h_{\delta}$ is bounded below by a smooth metric. Hence for any $(n,0)$-from $f$ with $f\wedge \ov{f} h_{\delta} \in L^1_{loc}$, it can be inferred that $f\in L^2_{loc}$. See \cite[Lemma 2.3.22]{MaMa07} for details.
On $X_1$, we have a similar relationship.
We define $\mathcal{A}^{n,0}(\ov{Y}_1, F_D, h_{\delta})$ the set of holomorphic $(n,0)$-forms with values in the bundle $F\otimes \mathcal{O}(D)$ in the neighborhoods of $Y_1$ in $Y_2$ (not the neighborhoods in $X_2$). Similarly, we denote by $H^0(\ov{X}_1, K_X\otimes F\otimes \mathcal{O}(D))$ the set of holomorphic section of $K_X\otimes F\otimes \mathcal{O}(D)$ in the neighborhoods of $X_1$ in $X_2$. According to the above relationship, we have $ \mathcal{A}^{n,0}(\ov{Y}_1, F_D, h_{\delta}) \subset H^0(X_1, K_X\otimes F\otimes \mathcal{O}(D))$ and $H^0(\ov{X}_1, K_X\otimes F\otimes \mathcal{O}(D))\subset \mathcal{A}^{n,0}(\ov{Y}_1, F_D, h_{\delta})$.
Our key step is to show that the restriction map
$$
\mathcal{A}^{n,0}(Y_2, F_D, h_{\delta}) \rightarrow \mathcal{A}^{n,0}(\ov{Y}_1, F_D, h_{\delta}),
$$
has a dense image concerning the $L^2$ norms.

\begin{remark} [Smooth increasing convex function] \label{construction1}
We take a smooth increasing convex function $\tau(t)$ such that:
\begin{enumerate}
    \item $\tau(t): (-\infty, +\infty) \rightarrow (-\infty, +\infty)$,
    \item $\tau(t)=0$ if $t\leq \frac{1}{c_2-c_1}$ and $\tau(t)>0$ when $t> \frac{1}{c_2-c_1}$,
    \item $\int_{0}^{+\infty} \sqrt{\tau''(t)} dt = +\infty$.
\end{enumerate}
We set $\Psi = \tau(\frac{1}{c_2-\Phi})$, it is a psh exhaustion function on $X_2$ and $\Psi \equiv 0$ on $X_1$ by the construction.
\end{remark}

\begin{remark} \label{construction2}
For each non-negative integers $m \geq 0$, we define new metric on $F_D = F\otimes \mathcal{O}(D)$ from \eqref{con-metric}:
\begin{align*}
h_{\delta_1} &:= h_{\delta} e^{-\Psi}, \\
h_{\delta_m} &:= h_{\delta} e^{-m\Psi}.
\end{align*}
We define a complete K\"ahler metric $\ov{\omega}$ on $Y_2$ by
$$
\ov{\omega} := \omega_{c_2,p} + \xu \pa\dbar \Psi.
$$
Here $ \omega_{c_2,p}$ is the Poincar\'e type metric along $D$, recall that
$$
\omega_{c_2,p} = k_{c_2}\omega_{c_2} - \frac{1}{2}  \xu\sum\pa\dbar \log \log^2\|\sigma_i \|_i^2
$$
as defined in Definition \ref{pointype}. We choose a large positive constant $k_{c_2}$ to ensure $\omega_{c_2,p}$ is positive on $Y_2$. By Proposition \ref{hatom} and the above Remark \ref{construction1}, one can observe that $\ov{\omega}$ is complete on $Y_2$.
\end{remark}

\begin{remark} \label{eigenvalue}
We define a new curvature form $\xu \Theta_m := \xu \Theta_{F_D,h_{\delta}} + \xu \pa\dbar m \Psi$ of $(F_D, h_{\delta_m})$. We want to compare it with $\ov{\omega}$. One obtains
\begin{equation}
    \ov{\omega} = k_{c_2}\omega_{c_2} - \frac{1}{2} \xu\sum\pa\dbar \log \log^2\|\sigma_i \|_i^2 + \xu \pa\dbar \Psi,
\end{equation}
and
\begin{equation}
   \xu \Theta_m  = \xu \Theta_{F,h^F} -\xu \kappa \sum\pa\dbar \log(\log^2\|\sigma_i\|^2_i) + \xu \delta \pa\dbar \phi_1 + \xu \pa\dbar m \Psi.
\end{equation}
We know $\xu\Theta_{F,h^F}$ is positive with respect to $\omega_{c_2}$. So we can choose $\kappa,\delta$ small enough such that the eigenvalues of $\xu \Theta_m$ concerning $\ov{\omega}$ are all positive on the whole $Y_2$. More specifically, at each point $x\in Y_2$, we may choose a coordinate system that diagonalizes simultaneously the forms $\ov{\omega}$ and $\xu \Theta_m$, in such a way that
$$
\ov{\omega}(x) = \xu \sum_{1\leq j \leq n}dz_j \wedge d\ov{z}_j, \quad \xu \Theta_m(x) = \xu \sum_{1\leq j \leq n} \gamma_j dz_j \wedge d\ov{z}_j.
$$
We will show there exists a positive constant $\epsilon$ such that $\gamma_j > \epsilon$ holds on $Y_2$ for each $\gamma_j$.
\end{remark}

\begin{definition} [Inner product] \label{inner-product}
For any non-negative integer $m$ and any $\varphi,\psi \in L^{n,0}(Y_2, F_D, h_{\delta_m})$, we define the inner product
$$
(\varphi, \psi)_m := \int_{Y_2} \left \langle \varphi, \psi \right \rangle_{\ov{\omega}} h_{\delta_m} dV = \int_{Y_2} \left \langle \varphi, \psi \right \rangle_{\ov{\omega}} h_{\delta} e^{-m\Psi} dV,
$$
and $\|\varphi \|^2_m = (\varphi,\varphi)_m$. We denote the adjoint operator of $\dbar$ in $L^{n,q}(Y_2,F_D,h_{\delta_m})$ by $ \dbar^{\ast}_m$.
\end{definition}

The next lemma is very important for our proof.

\begin{lemma} [Uniform estimate] \label{uniform-estimate}
There exist a positive constant $M$ which is independent to $m$ such that for any $m \geq 0$ and $0\leq q \leq n$, we have the estimate
$$
\| \varphi\|^2_m \leq  M (\| \dbar\varphi\|^2_m + \| \dbar^{\ast}_m\varphi\|^2_m)
$$
provided $\varphi \in D^{n,q}_{\dbar} \cap D^{n,q}_{\dbar^{\ast}_m} \subset  L^{n,q}(Y_2,F_D,h_{\delta_m})$. Here $D^{n,q}_{\dbar}$ is the domain of definition of $\dbar$ in $ L^{n,q}(Y_2,F_D,h_{\delta_m})$, and $D^{n,q}_{\dbar^{\ast}_m}$ is similar.
\end{lemma}

\begin{proof}
Let $D_m = \delta_m + \dbar$ be the Chern connection associated with the inner product in Definition \ref{inner-product}. We denote the adjoint operator of $\delta_m$ and $\dbar$ in $L^{n,q}(Y_2,F_D,h_{\delta_m})$ by $\delta^{\ast}_m$ and $\dbar^{\ast}_m$ respectively. Let us define $\Delta'' := \dbar\dbar^{\ast}_m + \dbar^{\ast}_m\dbar$ and $\Delta' := \delta_m  \delta^{\ast}_m +  \delta^{\ast}_m \delta_m$.

Since $\ov{\omega}$ is a complete K\"ahler metric on $Y_2$, the classical Bochner--Kodaira--Nakano identity shows
$$
\Delta'' = \Delta' + [i\Theta_m, \Lambda_{\ov{\omega}}].
$$
If $\varphi \in \mathcal{C}^{\infty}_0(Y_2, \Lambda^{n,q}T^{\ast}Y \otimes F_D)$ be a smooth compact supported $F_D$-valued $(n,q)$-form. We have
\begin{align} \label{BKNineq}
    \|\dbar \varphi \|^2_m + \|\dbar^{\ast}_m \varphi \|^2_m \geq \int_{Y_2} \left \langle [i\Theta_m, \Lambda_{\ov{\omega}}]\varphi, \varphi \right \rangle_{\ov{\omega}} h_{\delta} e^{-m \Psi} dV.
\end{align}
By the above Remark \ref{eigenvalue}, we have
\begin{equation} \label{ovomega}
    \ov{\omega} = k_{c_2}\omega_{c_2} - \frac{\xu}{2} \sum \pa\dbar \log(\log^2\|\sigma_i \|_i^2) + \xu \pa\dbar \Psi,
\end{equation}
and
\begin{equation*}
   \xu \Theta_m  = \xu \Theta_{F,h^F} -\xu\kappa \sum\pa\dbar \log(\log^2\|\sigma_i\|^2_i) + \xu \delta \pa\dbar \phi_1 + \xu \pa\dbar m \Psi.
\end{equation*}
We can diagonalize simultaneously the Hermitian forms $\ov{\omega}$ and $\xu \Theta_m$.
We know $\xu \Theta_{F,h^F}$ is positive with respect to $\omega_{c_2}$, i.e., there exists a constant $\epsilon_1$ such that all eigenvalues of $\xu \Theta_{F,h^F}$ with respect to $\omega_{c_2}$ are bigger than $\epsilon_1$ on $Y_2$. If we let $\kappa = \frac{\epsilon_1}{4k_{c_2}}$, then
\begin{align*}
 \xu \Theta_m  &= \xu \Theta_{F,h^F} -\xu\frac{\epsilon_1}{4k_{c_2}} \sum \pa\dbar \log(\log^2\|\sigma_i\|^2_i) + \xu \delta \pa\dbar \phi_1 + \xu \pa\dbar m \Psi \\
 &\geq \epsilon_1 \omega_{c_2} - \xu\frac{\epsilon_1}{4k_{c_2}} \sum \pa\dbar \log(\log^2\|\sigma_i\|^2_i) + \xu \delta \pa\dbar \phi_1 + \xu \pa\dbar m \Psi\\
 &= \frac{\epsilon_1}{2k_{c_2}} (k_{c_2}\omega_{c_2} - \frac{\xu}{2} \sum \pa\dbar \log \log^2\|\sigma_i \|_i^2) +(\frac{\epsilon_1}{2} \omega_{c_2} + \xu \delta \pa\dbar \phi_1) + \xu \pa\dbar m \Psi.
\end{align*}
Compare this with formula \eqref{ovomega}, we can arrange $\delta$ small enough such that the eigenvalues of $\xu \Theta_m$ with respect to $\ov{\omega}$ are all positive on the whole $Y_2$. Hence there exist a positive constant $M_0$ on $Y_2$, independent to $m$, so that
$$
\left \langle [i\Theta_m, \Lambda_{\ov{\omega}}]\varphi, \varphi \right \rangle_{\ov{\omega}} \geq M_0 |\varphi|^2.
$$
So if we plug this back into formula \eqref{BKNineq} above, as a consequence, we get the desired uniform estimate
$$
\| \varphi\|^2_m \leq  M (\| \dbar\varphi\|^2_m + \| \dbar^{\ast}_m\varphi\|^2_m)
$$
for  $\varphi \in \mathcal{C}^{\infty}_0(Y_2, \lambda^{n,q}T^{\ast}Y \otimes F_D)$. Since the metric $\ov{\omega}$ is complete, the above estimate still holds provided  $\varphi \in D^{n,q}_{\dbar} \cap D^{n,q}_{\dbar^{\ast}_m} \subset  L^{n,q}(Y_2,F_D,h_{\delta_m})$.
\end{proof}

\begin{lemma} [Approximation lemma] \label{appro}
If $ \varphi \in \mathcal{A}^{n,0}(\ov{Y}_1, F_D, h_{\delta})$, then for any $\varepsilon >0$, there exist a $\tilde{\varphi} \in \mathcal{A}^{n,0}(Y_2, F_D, h_{\delta})$ such that $\|\tilde{\varphi}|_{Y_1} - \varphi \|^2_0 < \epsilon$.
\end{lemma}

\begin{proof}
According to the Hahn--Banach theorem (cf. \cite[Theorem 3.5]{Rudi91}, \cite[Lemma 4.3.1]{Horm90}), it suffices to show that if $u \in\ov{\mathcal{A}^{n,0}(\ov{Y}_1, F_D, h_{\delta})} \subset L^{n,0}(Y_1, F_D,h_{\delta})$ and
\begin{equation} \label{Y1or}
(u,f)_{Y_1} = \int_{Y_1} \left \langle u, f \right \rangle_{\ov{\omega}} h_{\delta} dV =0
\end{equation}
for any $f\in \mathcal{A}^{n,0}(Y_2, F_D,h_{\delta})$.
Then we have
\begin{equation} \label{Y2or}
(u,g)_{Y_1} = \int_{Y_1} \left \langle u, g \right \rangle_{\ov{\omega}}h_{\delta} dV =0
\end{equation}
provided $g\in \mathcal{A}^{n,0}(\ov{Y}_1, F_D, h_{\delta})$.

We change the definition of $u$ by setting $u=0$ on $Y_2\backslash Y_1$ and remain unchanged on $Y_1$, we denote it by $u'$. Since $\Psi \equiv 0$ on $X_1$ and therefore the above equality \eqref{Y1or} implies
$$
u' \perp \{ L^{n,0}(Y_2, F_D,h_{\delta_m}) \cap \ker \dbar\}
$$
for each $m$. Then we obtain $ u' \in \ov{\text{Im} \dbar^{\ast}_m} \subset L^{n,0}(Y_2, F_D,h_{\delta_m})$. According to the uniform estimate Lemma \ref{uniform-estimate}, we know
$$
\ov{\text{Im} \dbar^{\ast}_m} = \text{Im} \dbar^{\ast}_m.
$$
According to \cite{Horm65}, we acquire $u' = \dbar^{\ast}_m v_m$ for some $v_m \in L^{n,1}(Y_2, F_D, h_{\delta_m})$ with estimate
$$
\|v_m\|^2_m \leq C_1 \|u' \|^2_m \leq C_1 \| u'\|^2_0.
$$

%In general, we can show that \( \dbar^{\ast} w_m = \dbar^{\ast}_m v_m = u' \) provided the support of \( w_m \) or \( v_m \) is a subset of \( \overline{Y}_1 \).
%Indeed, by the second term of Remark $4.5$, we can assume both $\tau(\frac{1}{c_2-c_1})$ and its derivative $\tau'(\frac{1}{c_2-c_1})$ are equal $0$. Since we let $\Psi = \tau(\frac{1}{c_2-\Phi})$ and $\Phi =c_1$ on the boundary of $X_1$, we have $\Psi =0$ and $\dbar \Psi =0$ on the $\overline{X}_1$ (of course on $\overline{Y}_1$). Moreover we have $e^{-m \Psi}=1$ and $\dbar (e^{-m \Psi})=0$ on $\overline{Y}_1$ for any $m$. Recall that $\dbar^{\ast}_m$ (resp. $\dbar^{\ast}$) is the adjoint operator of $\dbar$ associate to the metric $h_{\delta}\cdot e^{-m\Psi}$ (resp. $h_{\delta}$).
 
We now define $w_m = e^{-m\Psi} v_m$. Next, we will demonstrate the existence of a weak limit of ${w_m}$ whose support lies within the closure of $Y_1$. This fact will be used in formula \eqref{64} below. Indeed we have:
$$
\|w_m\|^2_0 \leq  \|w_m\|^2_{-m} = \|v_m\|^2_m  \leq C_1 \| u'\|^2_0.
$$
Hence $\{w_m\}$ has a subsequence which is weakly convergent in $L^{n,1}(Y_2, F_D, h_{\delta})$, we denote the weak limit by $w$. We will show that $\text{supp} ~w \subseteq \ov{Y}_1$.
On the other hand, for every $\epsilon >0$, by the inequality $ \|w_m\|^2_{-m} \leq C_1 \| u'\|^2_0$, we have the inequality
$$
\int_{\{x\in Y_2 : \Psi>\epsilon\}} e^{m\Psi} \left \langle w_m, w_m \right \rangle_{\ov{\omega}} h_{\delta}dV \leq C_1 \|u' \|^2_0.
$$
Thus we have
$$
e^{m\epsilon}\int_{\{x\in Y_2 : \Psi>\epsilon\}}  \left \langle w_m, w_m \right \rangle_{\ov{\omega}} h_{\delta}dV \leq C_1 \|u' \|^2_0
$$
for each $m$. It follows that $\int_{\{x\in Y_2 : \Psi>\epsilon\}}  \left \langle w_m, w_m \right \rangle_{\ov{\omega}}h_{\delta} dV $ tends to zero and hence $w_m \rightarrow 0$ almost everywhere in $\{x\in Y_2 : \Psi>\epsilon\}$. As a consequence, the weak limit $w=0$ on $\{x\in Y_2 : \Psi>\epsilon\}$ for every $\epsilon>0$. In summary, we have
$$
\text{supp}~ w \subseteq \ov{Y}_1 \quad \text{and} \quad \dbar^{\ast} w= u'.
$$
In fact, for any corresponding compactly supported smooth form $\alpha$, one have
\begin{align} 
(w,\dbar \alpha)_{Y_2} &= \int_{Y_2} \left \langle w, \dbar \alpha \right \rangle_{\ov{\omega}} h_{\delta} dV \label{89}\\
&=\lim \int_{Y_2} \left \langle w_m, \dbar \alpha \right \rangle_{\ov{\omega}} h_{\delta} dV \nonumber\\
&=\lim \int_{Y_2} \left \langle  e^{-m\Psi} v_m, \dbar \alpha \right \rangle_{\ov{\omega}} h_{\delta} dV  \nonumber\\
 &=\lim \int_{Y_2} \left \langle v_m, \dbar \alpha \right \rangle_{\ov{\omega}} h_{\delta}\cdot e^{-m\Psi} dV   \nonumber\\
 &=\lim \int_{Y_2} \left \langle \dbar^{\ast}_m v_m, \alpha \right \rangle_{\ov{\omega}} h_{\delta}\cdot e^{-m\Psi} dV   \nonumber\\
  &=\lim \int_{Y_2} \left \langle u', \alpha \right \rangle_{\ov{\omega}} h_{\delta}\cdot e^{-m\Psi} dV   \nonumber\\
  &= \int_{Y_2} \left \langle u', \alpha \right \rangle_{\ov{\omega}} h_{\delta} dV = (u', \alpha)_{Y_2}.  \nonumber
\end{align}
The second equality is due to the fact that \( w \) is the weak limit of \( w_m \) in \( L^{n,1}(Y_2, F_D, h_{\delta}) \), and recall that \( u' = 0 \) outside of \( Y_1 \).

For any open neighborhood $H_1$ of $X_1$ in $X_2$. We can take a $\mathcal{C}^{\infty}$ function $\zeta$ on $X_2$ satisfying $0\leq \zeta \leq 1$, supp $\zeta \subseteq H_1$ and $\zeta =1$ on $X_1$.
For these $g \in \mathcal{A}^{n,0}(\ov{Y}_1, F_D, h_{\delta})$, we still have $\dbar(\zeta g) =0$ on $Y_1$. And we arrange $H_1$ very close to $X_1$ such that $g$ is defined on $H_1\backslash D$. Hence $\zeta g$ is defined on $Y_2$ and obviously belong in $L^{n,0}(Y_2, F_D, h_{\delta})$. So
\begin{align}
(u,g)_{Y_1} &= \int_{Y_1} \left \langle u, g \right \rangle_{\ov{\omega}} h_{\delta}dV = \int_{Y_2} \left \langle u', \zeta g \right \rangle_{\ov{\omega}}h_{\delta} dV \nonumber\\
&= \int_{Y_2} \left \langle \dbar^{\ast}w, \zeta g \right \rangle_{\ov{\omega}}h_{\delta} dV \nonumber\\
&= \int_{Y_2} \left \langle w, \dbar(\zeta g) \right \rangle_{\ov{\omega}}h_{\delta} dV  \label{64}\\
&= 0 \nonumber.
\end{align}
This confirms the equality \eqref{Y2or} and therefore completes the proof of Lemma \ref{appro}.
\end{proof}

\begin{definition} [Semi-norms]
Let $h$ be any smooth metric of $F_D = F\otimes \mathcal{O}(D)$ on the whole $X$. For a fixed real number $c$, the sublevel set $X_c$ is relatively compact in $X$. Let $K$ be a compact subset of $X_c$, we set
$$
|\varphi|_{K} := \underset{x\in K}{\text{sup}} \sqrt{\left \langle \varphi, \varphi \right \rangle_{\omega}h(x)}
$$
for $\varphi \in H^0(X_c, K_X \otimes F_D)$, where $\left \langle \varphi, \varphi \right \rangle_{\omega} h(x)$ be the pointwise norms and it is independent to $\omega$ because $\varphi$ is an $(n,0)$-form.
\end{definition}

We can find positive constants $M_2$ such that $h \leq M_2 h_{\delta}$ on $Y_c$, here $M_2$ is a  constant depends on $Y_c$. So using Cauchy's integral formula in each local coordinate $U_i$ with $U_i \cap K \neq \emptyset$, we have
\begin{align*}
 |\varphi|^2_{U_i\cap K}  &\leq M_3 \int_{U_i\cap K} |\varphi|^2 h dV\\
 & \leq M_2 M_3 \int_{U_i\cap K} |\varphi|^2 h_{\delta} dV \\
 & \leq M_2 M_3 \|\varphi \|_0^2.
\end{align*}
This shows we can find a positive constant $M$ depends on $X_c$ such that
$$
|\varphi|_K \leq M \|\varphi \|_0.
$$
In summary, we get the desired approximation.
\begin{lemma} \label{maxappro}
Let $X_1 \subset X_2$ be the pair of sublevel sets. Then for any holomorphic section $\varphi \in H^0(\ov{X}_1, K_X\otimes F\otimes \mathcal{O}(D))$ and for any $\epsilon >0$, there exists a section $\tilde{\varphi} \in H^0(X_2, K_X\otimes F\otimes \mathcal{O}(D))$ such that $|\tilde{\varphi} - \varphi|_{\ov{X}_1} < \epsilon$. Note that here $\varphi \in \mathcal{A}^{n,0}(Y_1, F_D, h_{\delta})$.
\end{lemma}

Now we can prove the vanishing of $H^1(X, K_X\otimes \mathcal{O}_X(D)\otimes F)$.

\begin{theorem} \label{q1}
Let $X$ be a weakly pseudoconvex K\"ahler manifold and $F$ be a positive line bundle on $X$. Let $D$ be a simple normal crossing divisor on $X$, here the number of irreducible components of the divisor may be infinite. We have
$$
H^1(X, K_X\otimes \mathcal{O}_X(D)\otimes F) = 0.
$$
\end{theorem}

\begin{proof}
Recall that we have established the vanishing of $H^1$ for $K_X\otimes \mathcal{O}_X(D)\otimes F$ on each $X_c$ for any real number $c$, specifically $H^1(X_c, K_X\otimes F\otimes \mathcal{O}(D)) = 0$. Let $\mathfrak{X}= \{ X_v\}_{v\geq0}$ and $\mathfrak{X}_v= \{ X_k\}_{k\leq v}$ be the covering of $X$ and $X_v$ respectively. Moreover, both $\mathfrak{X}$ and $\mathfrak{X}_v$ serve as Leray coverings for the sheaf $K_X\otimes F\otimes \mathcal{O}(D)$ on $X$ and $X_v$, respectively. We have
\begin{equation*}
H^1(X, K_X\otimes F\otimes \mathcal{O}(D)) = \check{H}^1(\mathfrak{X}, K_X\otimes F\otimes \mathcal{O}(D)),
\end{equation*}
and
\begin{equation} \label{last}
H^1(X_v, K_X\otimes F\otimes \mathcal{O}(D)) = \check{H}^1(\mathfrak{X}_v, K_X\otimes F\otimes \mathcal{O}(D)) =0,
\end{equation}
for each $v$.

For any $1$-cocycle $\sigma \in Z^1(\mathfrak{X}, K_X\otimes \mathcal{O}(D)\otimes F)$, let $\sigma_v$ be the restriction of $\sigma$ to $\mathfrak{X}_v$. Then it is obviously that $\sigma_v \in Z^1(\mathfrak{X}_v, K_X\otimes \mathcal{O}(D)\otimes F)$. According to the local vanishing \eqref{last}, there is a $0$-cochain $\alpha_v \in  C^0(\mathfrak{X}_v, K_X\otimes \mathcal{O}(D)\otimes F)$ such that $\delta \alpha_v = \sigma_v$. As an element of $ C^0(\mathfrak{X}_{v-1}, K_X\otimes \mathcal{O}(D)\otimes F)$, we have $\delta \alpha_v = \delta \alpha_{v-1}$, and hence $\alpha_v - \alpha_{v-1} \in Z^0(\mathfrak{X}_{v-1}, K_X\otimes \mathcal{O}(D)\otimes F)$, i.e., the holomorphic section on $X_{v-1}$, write as $\alpha_v - \alpha_{v-1} \in \Gamma(X_{v-1}, K_X\otimes \mathcal{O}(D)\otimes F)=H^0(X_{v-1}, K_X\otimes \mathcal{O}(D)\otimes F)$.
Now by the approximation Lemma \ref{maxappro}, for any $\epsilon>0$ we can find a $ \gamma \in \Gamma(X_{v}, K_X\otimes \mathcal{O}(D)\otimes F)$ so as to
$$
|\alpha_v - \alpha_{v-1} - \gamma|_{\ov{X}_{v-2}} < \epsilon.
$$
Let's start with $v=1$. For any $\epsilon>0$, we know there exists a ${\gamma}_1\in\Gamma(X_{1}, K_X\otimes \mathcal{O}(D)\otimes F)$ such that $|\alpha_1 - \alpha_0 - {\gamma}_1|_{\ov{X}_{-1}} < \epsilon$, here we choose a relativelly compact subset $X_{-1} \subset X_0$. We set $\lambda_1 := (\alpha_1 - {\gamma}_1)$, we know $ \lambda_1 \in C^0(\mathfrak{X}_{1}, K\otimes \mathcal{O}(D)\otimes F)$ since $\alpha_1 \in C^0(\mathfrak{X}_{1}, K\otimes \mathcal{O}(D)\otimes F)$ and ${\gamma}_1 \in \Gamma(X_{1}, K_X\otimes \mathcal{O}(D)\otimes F)$.
It is easy to see that $\delta \lambda_1 = \delta \alpha_1 = \sigma_1$.

Assume we have constucted a sequence $\lambda_0, \cdots, \lambda_k$ such that for any $0\leq j\leq k$
\begin{enumerate}
    \item $\lambda_j \in C^0(\mathfrak{X}_{j}, K\otimes \mathcal{O}(D)\otimes F)$ and $\lambda_0 = \alpha_0$,
    \item $\delta \lambda_j = \sigma_j$,
    \item $ \lambda_{j} - \lambda_{j-1}\in \Gamma(X_{j-1}, K_X\otimes \mathcal{O}(D)\otimes F)$ and $|\lambda_{j} - \lambda_{j-1}|_{\ov{X}_{j-2}} < \frac{1}{2^{j-1}}$.
\end{enumerate}
 We know there is a $0$-cochain $\alpha_{k+1} \in  C^0(\mathfrak{X}_{k+1}, K_X\otimes \mathcal{O}(D)\otimes F)$ such that $\delta \alpha_{k+1} = \sigma_{k+1}$. Thus, $\alpha_{k+1} - \lambda_{k} \in \Gamma(X_{k}, K_X\otimes \mathcal{O}(D)\otimes F)$.  For a given constant $\frac{1}{2^{k}}$, we know there exists a ${\gamma}_{k+1}\in\Gamma(X_{k+1}, K_X\otimes \mathcal{O}(D)\otimes F)$ such that $|\alpha_{k+1} - \lambda_{k} - \gamma_{k+1}|_{\ov{X}_{k-1}} < \frac{1}{2^{k}}$. We set $\lambda_{k+1} := \alpha_{k+1} - \gamma_{k+1}$ as desired.

Therefore, inductively we have a sequence $\{\lambda_v\}_{v\geq 0}$ such that
\begin{enumerate}
    \item $\lambda_v \in C^0(\mathfrak{X}_{v}, K\otimes \mathcal{O}(D)\otimes F)$ and $\lambda_0 = \alpha_0$,
    \item $\delta \lambda_v = \sigma_v$,
    \item $ \lambda_{v+1} - \lambda_{v}\in \Gamma(X_{v}, K_X\otimes \mathcal{O}(D)\otimes F)$ and $|\lambda_{v+1} - \lambda_{v}|_{\ov{X}_{v-1}} < \frac{1}{2^v}$.
\end{enumerate}
As a consequence, for any $v$,
$$
\lim_{u\geq v}\lambda_u = \lambda_v + \sum_{k\geq v}(\lambda_{k+1}-\lambda_k)
$$
defines an element of $C^0(\mathfrak{X}_{v}, K\otimes \mathcal{O}(D)\otimes F)$ when restricted to $\mathfrak{X}_v$. Moreover
$$
\lim_{u\geq v+1}\lambda_u = \lambda_{v+1} + \sum_{k\geq v+1}(\lambda_{k+1}-\lambda_k)
$$
defines the same element as $\lim_{u\geq v}\lambda_u$ when restricted to $\mathfrak{X}_{v}$. Thus we can define an element $\lambda$ of $C^0(\mathfrak{X}, K_X\otimes \mathcal{O}(D)\otimes F)$ by $\lambda = \lim_{v \to \infty}\lambda_v$. For any $v$, when restricted to $C^0(\mathfrak{X}_v, K_X\otimes \mathcal{O}(D)\otimes F)$ we have
$$
\delta (\lim_{u\geq v}\lambda_u) = \lim_{u\geq v} \delta\lambda_u = \sigma_v.
$$
Hence we have $\delta \lambda = \sigma$ and the proof is complete.
\end{proof}

\bigskip
%%%%%%%%%%%% Authors' addresses %%%%%%%%%%%%%
\address{ % First Author
Graduate School of Mathematical \endgraf
The University of Tokyo \endgraf
3--8--1 Komaba, Meguro-Ku, Tokyo \endgraf
Japan
}
{zouyongpan@gmail.com}
%
%\address{% Second Author
%Mathematical Institute \\
%Tohoku University \\
%Sendai 980-8578 \\
%Japan
%}
%{author2@math.tohoku.ac.jp}
%
%%start a new line
%\address{% Third Author
%Mathematical Institute \\
%Tohoku University \\
%Sendai 980-8578 \\
%Japan
%}
%{author3@math.tohoku.ac.jp}
%%
%\address{% Fourth Author
%Mathematical Institute \\
%Tohoku University \\
%Sendai 980-8578 \\
%Japan
%}
%{author4@math.tohoku.ac.jp}

\end{document}